\documentclass[a4paper,11pt]{amsart}
\usepackage[utf8]{inputenc}
\usepackage{amsthm,amsmath,amssymb,anysize,enumerate,tikz-cd,dsfont}
\usepackage[mathscr]{eucal}
\usepackage[cmtip,all]{xy}
\usepackage{hyperref}
\usepackage{cite}
\marginsize{2cm}{2cm}{2cm}{2cm}

\renewcommand{\AA}{\mathds A}
\newcommand{\CC}{\mathds C}
\newcommand{\DD}{\mathds D}
\newcommand{\gm}{{\mathds G}_m}
\newcommand{\gmt}{{\mathds G}_{m,t}}
\newcommand{\GG}{\mathds G}
\newcommand{\LL}{\mathbb L}
\newcommand{\NN}{\mathds N}
\newcommand{\PP}{\mathds P}
\newcommand{\RR}{\mathds R}
\newcommand{\QQ}{\mathds Q}

\newcommand{\ZZ}{\mathds Z}

\renewcommand{\L}{\mathbf L}
\newcommand{\R}{\mathbf R}

\newcommand{\cA}{\mathcal A}

\newcommand{\cD}{\mathcal D}
\newcommand{\cE}{\mathcal E}
\newcommand{\cF}{\mathcal F}
\newcommand{\cH}{\mathcal H}
\newcommand{\cI}{\mathcal I}
\newcommand{\cJ}{\mathcal J}
\newcommand{\cK}{\mathcal K}
\newcommand{\cL}{\mathcal L}
\newcommand{\cM}{\mathcal M}
\newcommand{\cN}{\mathcal N}
\newcommand{\cO}{\mathcal O}
\newcommand{\cR}{\mathcal{R}}
\newcommand{\cT}{\mathcal T}

\newcommand{\cX}{\mathcal X}

\newcommand{\cQ}{\mathcal Q}

\newcommand{\IC}{\mathit{IC}}

\newcommand{\fe}{\mathfrak{e}}
\newcommand{\fp}{\mathfrak{p}}

\newcommand{\cRint}{\cR^{\operatorname{int}}}
\newcommand{\Db}{\text{D}^\text{b}}
\newcommand{\Dbc}{\text{D}_{\text{c}}^\text{b}}

\newcommand{\Dbrh}{\text{D}_{\text{rh}}^\text{b}}
\newcommand{\dd}{\partial}
\newcommand{\Firr}{F^{\operatorname{irr}}}
\newcommand{\GF}{G^{F_\bullet}}

\newcommand{\ttau}{{}^\tau\!}
\newcommand{\ttheta}{{}^\theta\!}
\newcommand{\wh}{\widehat}

\newcommand{\ucero}{\underline{0}}
\newcommand{\uuno}{\underline{1}}
\newcommand{\ulambda}{\underline{\lambda}}
\newcommand{\uy}{\underline{y}}

\newcommand{\ra}{\rightarrow}
\newcommand{\hra}{\hookrightarrow}
\newcommand{\lra}{\longrightarrow}

\newcommand{\can}{\operatorname{can}}
\newcommand{\Conv}{\operatorname{Conv}}
\newcommand{\diag}{\operatorname{diag}}
\newcommand{\FL}{\operatorname{FL}}
\newcommand{\Gr}{\operatorname{Gr}}
\newcommand{\Hom}{\operatorname{Hom}}

\newcommand{\id}{\operatorname{id}}
\newcommand{\im}{\operatorname{im}}
\newcommand{\inv}{\operatorname{inv}}
\newcommand{\MHM}{\operatorname{MHM}}
\newcommand{\IrrMHM}{\operatorname{IrrMHM}}
\newcommand{\loc}{\operatorname{loc}}
\newcommand{\Mod}{\operatorname{Mod}}
\newcommand{\MTM}{\operatorname{MTM}}

\newcommand{\Spec}{\operatorname{Spec}}
\newcommand{\whloc}{\wh{\operatorname{loc}}}
\newcommand{\Sred}{{'\!}S}
\newcommand{\Pred}{{'\!}P}

{\theoremstyle{definition}
\newtheorem{defi}{Definition}[section]
\newtheorem{assump}[defi]{Assumptions}}
{\theoremstyle{remark}
\newtheorem{rem}[defi]{Remark}}
\newtheorem{thm}[defi]{Theorem}
\newtheorem{prop}[defi]{Proposition}
\newtheorem{lemma}[defi]{Lemma}
\newtheorem{coro}[defi]{Corollary}

\begin{document}

\title{Irregular Hodge filtration of some confluent hypergeometric systems}
\author{Alberto Castaño Domínguez}
\author{Christian Sevenheck}
\thanks{The authors are partially supported by the project SISYPH: ANR-13-IS01-0001-01/02 and DFG grant SE 1114/5-1.}
\email{alberto.castano@usc.es}
\email{christian.sevenheck@mathematik.tu-chemnitz.de}
\address{Instituto de Matemáticas, Universidade de Santiago de Compostela, 15782 Santiago de Compostela (Spain); Fakultät für Mathematik, Technische Universität Chemnitz, 09107 Chemnitz (Germany)}
\address{Fakultät für Mathematik, Technische Universität Chemnitz. 09107 Chemnitz (Germany)}

\begin{abstract}
We determine the irregular Hodge filtration, as introduced by Sabbah, for the purely irregular hypergeometric $\cD$-modules. We obtain in particular a formula for the irregular Hodge numbers of these systems.
We use the reduction of hypergeometric systems from GKZ-systems as well as comparison results to
Gau\ss-Manin systems of Laurent polynomials via Fourier-Laplace and Radon transformations.
\end{abstract}

\maketitle

\section{Introduction}\label{sec:Introduction}

The aim of this paper is to compute some Hodge theoretic invariants of certain classical differential systems in one variable. These are the so-called irregular Hodge numbers, which have been introduced recently by Sabbah \cite{Sa15}.
They are called irregular
because they are attached to differential systems which may have irregular singular points, a feature that is excluded for classical variations
of Hodge structures as well as for the more general Hodge modules. The very definition of these numbers rely on the theory of mixed twistor $\cD$-modules of T.~Mochizuki (\hspace{-.5pt}\cite{Mo13}). Twistor $\cD$-modules generalise Hodge modules, in the sense that the underlying $\cD$-module of a twistor $\cD$-module
can have irregular singularities. In particular, one can define a version of the Fourier-Laplace transformation functor within the category of mixed twistor $\cD$-modules, which is impossible for mixed Hodge modules.
The drawback of this generalization is that one cannot directly assign a filtration to a twistor $\cD$-module, and hence
it is not easy to attach numerical invariants like Hodge numbers to it. In the above mentioned paper \cite{Sa15}, Sabbah constructs
an intermediate category between mixed Hodge modules and mixed twistor $\cD$-modules (called irregular mixed Hodge modules) which is on the one hand sufficiently large to be stable under
all relevant operations that are defined for twistor $\cD$-modules (in particular, the Fourier-Laplace transformation), but which allows one to define a filtration,
called irregular Hodge filtration, for each of its objects. The construction is related to the earlier papers \cite{Yu, SaEsYu, Sa14}, where for certain
projective morphisms $f:X\rightarrow \PP^1$, a (rationally indexed) filtration was introduced on the twisted de Rham complex
$(\Omega^\bullet(*D), d+df\wedge)$, where $D\subset X$ is a certain normal crossing boundary divisor such that $f_{|X\backslash D}$ yields a regular function.
Instead of considering meromorphic differential forms, one can also use the subcomplex of so-called $f$-adapted logarithmic forms (also
 called Kontsevich complex, see \cite{KKP2})
$\Omega^\bullet_f$, which consists of
forms such that the exterior product with $df$ is still logarithmic along $D$.
From an $E_1$-degeneration property of the corresponding spectral sequence proved in \cite{SaEsYu}, one obtains a filtration on the twisted de Rham cohomology, called irregular Hodge filtration
of the regular function $f: X\backslash D \rightarrow \AA^1$. We refer to \cite{SaEsYu} for more details.

Notice also that the paper \cite{KKP2} gives three definitions of so-called Landau-Ginzburg Hodge numbers associated to a family $f:X\rightarrow \PP^1$,
one of them being $\dim H^p(\Omega^q_f)$. Conjecturally these three definitions coincide, but this seems to require some more assumptions (see \cite{Shamoto} and \cite{Sa18} for some partial results). Although these Hodge numbers have only integer indices, they are closely related to the dimensions of the graded parts of the filtration
from \cite{SaEsYu}. Ultimately, according to \cite{KKP2} and following predictions from homological mirror symmetry, one hopes for a correspondence between the Hodge numbers of some, say, projective varieties entering in the A-model and the irregular Hodge numbers of its (Landau-Ginzburg) B-model.

For the moment, there are quite a few examples where the irregular Hodge filtration can actually be computed.
A central result of \cite{Sa15} states that rigid irreducible $\cD$-modules on the projective line
underlie objects of the category of irregular mixed Hodge modules, and consequently admit a unique irregular Hodge filtration, provided that their formal local monodromies are unitary.
Rigid $\cD$-modules are particularly interesting since they can be algorithmically constructed from simple objects by an algorithm due to Arinkin and Katz (cf. \cite{Arin}).
Among the most studied and best understood examples of such rigid $\cD$-modules are the classical hypergeometric $\cD$-modules. In the regular case (which corresponds to classical variations of complex Hodge structures), Fedorov has given in \cite{Fe} a closed formula for the Hodge numbers (without computing the Hodge filtration itself however) conjectured by Corti and Golyshev in \cite{CorGol} using the work \cite{DettSa} of Dettweiler and Sabbah.

In the present paper, we consider the case of purely irregular hypergeometric modules. Our principal result, Theorem \ref{thm:HodgeData}, completely determines the irregular Hodge filtration and gives a very simple formula for the irregular Hodge numbers, which is in some sense
similar to the shape of Fedorov's formula. The main ingredients are a reduction process (as explained in \cite{BMW}) to obtain classical hypergeometric $\cD$-modules from
some higher dimensional ones, the so-called GKZ-systems, techniques from \cite{Reich2} and \cite{ReiSe,ReiSe2} (going back to \cite{Sa8}) to understand
Hodge module structures on these GKZ-systems as well as a quite explicit solution to the so-called Birkhoff problem that is
inspired by calculations in toric mirror symmetry (see again \cite{ReiSe} as well as \cite{DS2} and also \cite{dGMS}).

Since the first version of this paper was written, Hodge invariants for hypergeometric systems have been considered
in some other articles. First, the formula for the Hodge numbers of purely irregular systems
has also been obtained by Sabbah and Yu in the final version of
\cite{Sa15} by different means. Another approach to Fedorov's formula for the case of regular systems (giving more precise information
on the various Hodge invariants) has been given by N.~Martin (see \cite{Mar18}).
Moreover, we have considered the case of hypergeometric operators of type $(n,1)$ (see Definition
\ref{def:ClassicHyp} below)
in a common paper with Th.~Reichelt (\hspace{-.5pt}\cite{SevCastReich}) using the computation of Hodge filtrations on GKZ-systems
from \cite{ReiSe3}. Finally, Sabbah and Yu have recently given a complete formula for the irregular Hodge numbers
for all confluent hypergeometric systems in \cite{SaYu18}. However, for the moment the irregular Hodge filtration itself
is not determined in the general case.

Let us recall some notation that will be adopted throughout the paper. For a smooth complex algebraic variety $X$, we write $\cD_X$ for the sheaf of algebraic differential operators on $X$. If $X$ is affine, we sometimes switch freely between sheaves of $\cD_X$-modules and modules
of global sections. We will denote the abelian categories of holonomic resp. regular holonomic $\cD_X$-modules by $\Mod_\text{h}(\cD_X)$ resp. $\Mod_\text{rh}(\cD_X)$, and analogously with the respective bounded derived categories.
For a morphism $f:X\rightarrow Y$, we denote the direct resp. inverse image functors for $\cD$-modules as usual by $f_+$ resp. $f^+$ (see
\cite{Hotta} for a thorough discussion of these and related functors).
We put $\GG_m:=\text{Spec\,} \CC[t,t^{-1}]$; if we want to fix a coordinate on this one-dimensional torus, we also write $\gmt$.

We denote by $\cRint_{\AA_z^1\times X}$ the
sheaf of Rees rings on $X$ (with integrable structure), that is, the subsheaf of non-commutative algebras of $\cD_{\AA^1_z\times X}$ generated by $zp^*\Theta_X$ and $z^2\partial_z$, where $p:\AA^1_z\times X \rightarrow X$ is the projection. If
$(x_1,\ldots,x_n)$ are local coordinates on $X$, then $\cRint_{\AA_z^1\times X}$ is locally given by
$\cO_{\AA^1_z\times X}\langle z^2\partial_z, z\partial_{x_1},\ldots,z\partial_{x_n}\rangle$.
Occasionally, we will also need the sheaf $\cR_{\AA^1_z\times X}$, which is generated
by $zp^*\Theta_X$ only, i.e., locally given by
$\cO_{\AA^1_z\times X}\langle z\partial_{x_1},\ldots,z\partial_{x_n}\rangle$.
We let $\MHM(X)$ be the abelian category of algebraic mixed Hodge modules (see \cite{SaitoMHM}) on $X$. We consider moreover the category $\MHM(X,\CC)$ of \emph{complex} mixed Hodge modules (see, e.g., \cite[Definition 3.2.1]{DettSa}).
As an example, if $X$ is the algebraic torus $\GG_m^d$ with coordinates $y_1,\ldots,y_d$, then
for any $\alpha=(\alpha_1,\ldots,\alpha_d)\in \RR^d$, the free $\cO_{\GG_m^d}$-module of rank $1$
$$
\cO_{\GG_m^d}^\alpha := \cD_{\GG_m^d}/(y_k\partial_{y_k}+\alpha_k+1)_{k=1,\ldots,d}
$$
(see also Definition \ref{def:Oalpha} below) underlies an object in $\MHM(\GG_m^d,\CC)$.

On the other hand, $\MTM(X)$  denotes the abelian category of algebraic mixed twistor $\cD$-modules on $X$ (see \cite{Mo13}).
The category $\MTM^{\text{int}}(X)$ consists of those mixed twistor $\cD$-modules where the underlying $\cR_{\AA^1_z\times X}$-modules
have an integrable structure, i.e., where they are modules over $\cRint_{\AA_z^1\times X}$.

The category $\text{IrrMHM}(X)$, as defined in \cite[Def. 2.52]{Sa15}, is a certain subcategory of $\MTM^{\text{int}}(X)$
(actually of a variant, called ${_\iota\!}\MTM^{\text{int}}(X)$, which is shown to be equivalent to $\MTM^{\text{int}}(X)$ in [ibid., \S~1])
consisting of objects $\wh\cM$ that satisfy certain properties. The first of them is that the object ${^\theta}\wh\cM$ obtained from $\wh\cM$ by substituting $z\theta$ for $z$ is still an object of ${_\iota\!}\MTM^{\text{int}}(\ttheta X)$, where $\ttheta X=\mathds{G}_{m,\theta}\times X$. That is a remarkable assumption, having its origin in \cite{HS1}, but we have to impose a further property, namely that such rescaled objects must have a certain tame behaviour when $\theta$ goes to infinity (or at the origin of $\tau=1/\theta$, in other words), related to regularity along $\{\tau=0\}\times\AA_z^1\times X$. This two conditions are, essentially, what we ask for a mixed twistor $\cD$-module on $X$ to be a mixed twistor-rescaled $\cD$-module on $X$ (\hspace{-.5pt}\cite[Def. 2.50]{Sa15}). The last condition is a $z$-grading property appearing after identifying $\tau$ with $z$ (cf. [ibid., Def. 2.26]). For more details, see [ibid., \S~2].

Although the construction of this category may seem rather involved, its main feature is that the $\cD_X$-module $\cM$ associated to an object $\widehat{\cM}$ in $\IrrMHM(X)$ carries a good filtration $F_\bullet^{\text{irr}} \cM$, indexed by $\RR$, called the irregular Hodge filtration, which in turn behaves well with respect to several functorial operations. Notice however that this filtration, contrarily to the case of mixed Hodge modules, is not part of the definition of an object of $\IrrMHM(X)$. Very roughly, it can be thought of as defined by the intersection of the canonical $V$-filtration
along $\tau=0$ (or rather the filtration induced on the restricted object when $\tau=z$) with the $z$-adic filtration on $\wh\cM$. In particular, the jumping indices of the irregular Hodge filtration are of the form
$\{\alpha+k\,|\,k\in \ZZ\}$ for a certain finite set of real numbers $\alpha$.

\textbf{Acknowledgements.} We would like to thank Takuro Mochizuki, Thomas Reichelt and Claude Sabbah for their interest in our work and for many stimulating discussions. We are grateful to the anonymous referee for the many valuable comments and remarks. We thank the Max Planck Institute for Mathematics in the Sciences, where a significant part of the work presented here has been carried out.

\section{Hypergeometric modules and dimensional reductions}\label{sec:GKZHorn}
\label{sec:reduction}
In this section we will introduce two different kinds of hypergeometric $\cD$-modules: classical and GKZ. We will dedicate more time to classical ones, since they form one of the main objects of study of the paper, and will end by showing the relation between both types, which will be useful for us in the future. We state one of our main results (Theorem \ref{thm:HypIrrMHM}), which is proven in the next section.

\begin{defi}\label{def:ClassicHyp}
Let $(n,m)\neq(0,0)$ be a pair of nonnegative integers, and let $\alpha_1,\ldots,\alpha_n$ and $\beta_1,\ldots,\beta_m$ be elements of $\CC$. The (classical) hypergeometric $\cD$-module of type $(n,m)$ associated with the $\alpha_i$ and the $\beta_j$ is defined as the quotient of $\cD_{\gm}$ by the left ideal generated by the so-called hypergeometric operator
$$\prod_{i=1}^n(t\dd_t-\alpha_i)-t\prod_{j=1}^m(t\dd_t-\beta_j).$$
We will denote it by $\cH(\alpha_1,\ldots,\alpha_n;\beta_1,\ldots,\beta_m)$, or in an abridged way, $\cH(\alpha_i;\beta_j)$.
\end{defi}

In this paper we will be mostly concerned with hypergeometric $\cD$-modules of type $(n,0)$.

\begin{rem}\label{rem:BasicHyp}
The excluded type $(n,m)=(0,0)$ corresponds to the punctual delta $\cD_{\gm}$-module on $\gm$ $\cH\left(\emptyset;\emptyset\right)=\cD_{\gm}/(1-t)$.

On the other hand, if we denote the Kummer $\cD$-module $\cD_{\gm}/(t\dd_t-\eta)$ by
$\cK_\eta$, for any fixed complex number $\eta$, then $\cH(\alpha_i;\beta_j) \otimes_{\cO_{\gm}}\cK_\eta\cong \cH(\alpha_i+\eta;\beta_j+\eta)$. In particular, an overall integer shift of the parameters gives us an isomorphic $\cD$-module.

Every hypergeometric $\cD$-module has Euler characteristic -1 (cf. \cite[Lem. 2.9.13]{Ka}). For $n\neq m$, no hypergeometric $\cD$-module of type $(n,m)$ has singularities on $\gm$. If $n>m$ (resp. $m>n$), they have a regular singularity at the origin (resp. infinity) and an irregular singularity at infinity (resp. the origin) of irregularity one and slope $1/|n-m|$ of multiplicity $|n-m|$ (cf. \cite[Prop. 2.11.9]{Ka}). For $n=m$, the hypergeometric $\cD$-modules of type $(n,n)$ are regular, with singularities only at the origin, infinity and 1.
\end{rem}

\begin{prop}[\textbf{Irreducibility}]\label{prop:IrredHyp}
\emph{(cf. \cite[Prop. 2.11.9, 3.2]{Ka})} Let $\cH:=\cH(\alpha_i;\beta_j)$ be a hypergeometric $\cD$-module. It is irreducible if and only if for any pair $(i,j)$ of indices, $\alpha_i-\beta_j$ is not an integer.
\end{prop}

\begin{rem}
Assuming the irreducibility, we have a result which is stronger than the second paragraph of Remark \ref{rem:BasicHyp}. Namely, the isomorphism class of $\cH(\alpha_i;\beta_j)$ depends only on the classes modulo $\ZZ$ of the $\alpha_i$ and the $\beta_j$ (point 1 of \cite[Prop. 3.2]{Ka}), so we can choose such parameters on a fundamental domain of $\CC/\ZZ$.
\end{rem}

\begin{prop}[\textbf{Rigidity}]\label{prop:RigidHyp}
Let $\cH:=\cH(\alpha_i;\beta_j)$ be an irreducible hypergeometric $\cD$-module of type $(n,m)$, where $n\geq m$, and let $\cM$ be another irreducible $\cD_{\gm}$-module of Euler-Poincaré characteristic -1 which has no singularities outside $\{0,1,\infty\}$. Assume that
\begin{itemize}
\item $\cH\otimes\CC((t))\cong\cM\otimes\CC((t))$.
\item $\cH\otimes\CC((1/t))\cong\cM\otimes\CC((1/t))$.
\item If $n=m$, assume further that $\cM$ has a regular singularity at 1.
\end{itemize}
In that case, $\cH$ and $\cM$ are isomorphic.
\end{prop}
\begin{proof}
$\cH$ and $\cM$ being irreducible, both of them coincide with the middle extension of their restriction to $\gm\setminus\{1\}$. Assume first that $n=m$, that is, $\cH$ is regular. Since $\cM$ has characteristic -1 and regular singularities at the origin, one and infinity, by the formula for the Euler characteristic \cite[Thm. 2.9.9]{Ka}, its formal local monodromy at 1 must be a pseudoreflection. Then we can apply [ibid., Rigidity Thm. 3.5.4] and we are done.

If $n>m$, $\cM$, like $\cH$, has a regular singularity at the origin and an irregular singularity at infinity. By the same formula for the Euler characteristic of $j_{\dag+}j^+\cM$ (denoting by $j$ the inclusion $\gm\setminus\{1\}\hra\gm$), $\cM$ cannot have more singularities in $\PP^1$ apart from zero and infinity. In that case we apply [ibid., Rigidity Thm. bis 3.7.3].
\end{proof}

Consider the case in which $\cH=\cH(\alpha_i;\beta_j)$ is of type $(n,n)$, i.e., such that it is regular.
From the rigidity property of the last Proposition, one concludes by \cite[Cor. 8.1]{Si2} that the restriction
of $\cH$ to $\gm \backslash \{1\}$ underlies a complex polarizable variation of Hodge structures. The next statement, conjectured by Corti and Golyshev (cf. \cite[Conj. 1.4]{CorGol}) and proved by Fedorov (see \cite[Thm. 1]{Fe}), gives important information on its Hodge invariants.
\begin{prop}[\textbf{Hodge numbers (regular case)}]\label{prop:Fedorov}
Let $\cH=\cH(\alpha_i;\beta_j)$ be an irreducible hypergeometric $\cD$-module of type $(n,n)$. Assume that the $\alpha_i$ and the $\beta_j$ are increasingly ordered real numbers, lying in the interval $[0,1)$. Set
$$
\rho(k)=\left|\{i=1,\ldots,n\,:\,\beta_i<\alpha_k\}\right|-k,
$$
for $k=1,\ldots,n$.
Then the Hodge numbers of $\cH$ are, up to an overall shift,
$$
h^p=\left|\rho^{-1}(p)\right|=\left|\{k=1,\ldots,h\,:\, \rho(k)=p\}\right|.
$$
\end{prop}

This last result is the analogous one, in the regular case, to Theorem \ref{thm:HodgeData}, and served as the main motivation to start this project. Notice that this formula slightly differs from the one given in \cite{Fe}. This is because Fedorov considers the dual connection to ours by working with the space of solutions instead of with that of horizontal sections (cf. \cite[p. 10, proof of Lem. 3.6]{Mar18}). Now that we have seen part of the behaviour of classical hypergeometric $\cD$-modules, let us continue with the other family mentioned above.

\begin{defi}\label{def:GKZ}
Let $n\geq m$ two positive integers, and let $d=n-m$. Let $\beta\in\CC^d$ be a vector and let $A=(a_{ij})\in\text{M}(d\times n,\ZZ)$ be an integer matrix.
Consider the $n$-dimensional torus $\GG_m^n$ with coordinates $\lambda_1,\ldots,\lambda_n$. We define the Euler operators $E_i=\sum_ja_{ij}\lambda_j\dd_{\lambda_j}$, for $i=1,\ldots,d$, and the toric ideal
$$
I_A:=\left(\dd_\lambda^u-\dd_\lambda^v : Au=Av\right) \subset\CC[\partial_{\lambda_1},\ldots,\partial_{\lambda_n}]\subset D_{\AA^n}.
$$
Then, the GKZ-hypergeometric $\cD$-module (or system) is
$$\cM_A^\beta:=\cD_{\GG_m^n}/(\cI_A+(E_i-\beta_i:i=1,\ldots,d)_{i=1,\ldots,d}),$$
where $\cI_A$ is the sheafified version of the toric ideal $I_A$.
\end{defi}
Usually the module $\cM_A^\beta$ is defined to be an element in $\textup{Mod}(\cD_{\AA^n})$. However, we will later work
only with the restriction of such an object to $\GG_m^n$. In order to avoid the usage of the functor $j^+$ (where $j:\GG_m^n\hookrightarrow \AA^n$)
each time that we need to refer to $\cM_A^\beta$, we use this slightly non-standard definition.

A classical hypergeometric $\cD$-module can be considered as a dimensional reduction of a certain GKZ-system. We will describe this procedure in some more detail now,
because it allows us to apply some of the many known results on GKZ-systems to classical hypergeometric $\cD$-modules.

\begin{prop}\label{prop:hypGKZ}
Let $m$ be a positive integer. Let $A\in\operatorname{M}((m-1)\times m,\ZZ)$ be an integer matrix of rank $m-1$ and take $\kappa\in\CC^m$ such that $\kappa_1=0$. Consider the inclusion $\iota:\gm\hra\GG_m^m$ given by $t\mapsto(t,1,\ldots,1)$, and let $B=(b_1,\ldots,b_m)^{\text{\emph{t}}}\in\ZZ^m$ be a Gale dual of $A$, that is, an integer column matrix which generates $\ker_\QQ A$. Assume moreover that $b_1=1$. Put
$$
\eta:=\prod_{i=1}^mb_i^{b_i}.
$$
Then we have $h_{\eta}^+\cH(\alpha_i;\beta_j)\cong \iota^+\cM_A^{A\kappa}$, where $h_\eta$ is the automorphism of $\gm$ given by $t\mapsto\eta t$, and the unordered sets of parameters $\alpha_i$ and $\beta_j$, counted with multiplicities, are
$$(\alpha_i)=\left(\frac{k-\kappa_j}{b_j}\,:\,b_j>0, k=0,\ldots,b_j-1\right),$$
$$(\beta_i)=\left(\frac{k-\kappa_j}{b_j}\,:\,b_j<0, k=0,\ldots,-b_j-1\right).$$
\end{prop}
\begin{proof}
On one hand, $\cM_A^{A\kappa}$ is not only a GKZ-hypergeometric $\cD$-module, but also the restriction to $\gm^n$ of a lattice basis binomial $\cD_{\AA^m}$-module (cf. \cite[Def. 1.2]{BMW}, noting that the assumption on the columns of $B$ is not needed for the definition). This is because $A$ being of rank $m-1$ implies that the toric ideal $I_A$ coincides with the lattice basis ideal
$$I(B):=\left(\dd_\lambda^{w_+}-\dd_\lambda^{w_-} : w=w_+-w_- \text{ is a column of } B\right).$$
(In fact this holds for any complete intersection ideal, but here the argument is simpler.)

On the other hand, the expression we have given for the parameters of $\cH(\alpha_i;\beta_j)$ follows from applying the definition of Horn hypergeometric $\cD$-modules given in [ibid., Def. 1.1] for a column matrix, up to the same caveat above about the columns on $B$ (in fact normalized ones, but all kinds of Horn $\cD$-modules defined in loc. cit. are equal once restricted to the torus $\GG_{m,t}$), and comparing it with Definition \ref{def:ClassicHyp}.

Now let $j:\gm\hookrightarrow\AA^1$ be the canonical inclusion. The isomorphism given in [ibid., Thm. 1.4] relates lattice basis binomial $\cD$-modules to Horn hypergeometric ones. However, due to the previous discussions, we obtain the isomorphism in the statement just by applying $j^+$ to both sides of it.
\end{proof}

Note that the choice of $\kappa_1$ and $b_1$ in the statement of the Proposition force $\alpha_1$ to vanish. However, by Remark \ref{rem:BasicHyp}, any other hypergeometric $\cD$-module can be built from one of this form just by tensoring with a suitable Kummer $\cD$-module.

In our study of Hodge theoretic properties of hypergeometric $\cD_{\gm}$-modules below, we need to go in some sense in the opposite direction: Given sets $\{\alpha_i\}$, $\{\beta_j\}$, we would like to determine a matrix $A$ and a parameter vector $\beta$ such that the module $\cH(\alpha_i;\beta_j)$ can be obtained as an inverse image of the GKZ-system $\cM_A^\beta$. Although such a pair $(A,\beta)$ is not unique, a systematic way of constructing it can be formulated as follows.
\begin{coro}\label{coro:hypGKZ}
Let $\cH(\alpha_i;\beta_j)$ be a hypergeometric $\cD_{\gm}$-module of type $(n,m)$ with $n>0$ and $\alpha_1=0$. Let $A\in\operatorname{M}((m+n-1)\times(n+m),\ZZ)$ given by
$$A=\left(\begin{array}{c|c|c}
\uuno_m & \ucero_{m\times(n-1)} & \operatorname{Id}_m\\[3pt]
\hline
 & & \vspace{-10pt}\\
\uuno_{n-1} & -\operatorname{Id}_{n-1} & \ucero_{(n-1)\times m}\end{array}
\right),$$
and let $\beta=(\beta_1,\ldots,\beta_m,\alpha_2,\ldots,\alpha_n)^{\operatorname{t}}$. (Here and later in this paper,
we write $\underline{x}_k$ for a column vector with $k$ rows, where all entries contain the value $x$.) Let $\iota:\gm\ra\GG_m^{n+m}$,
given by
$t\mapsto(t,1\ldots,1)$.
Then
$$\cH(\alpha_i;\beta_j)\cong\iota^+\cM_A^\beta.$$
\end{coro}
\begin{proof}
The statement is an easy consequence of the Proposition, taking $B=(1,\overset{(n)}\ldots,1,-1,\overset{(m)}\ldots,-1)^{\operatorname{t}}$ and $\kappa=(0,-\alpha_2,\ldots,-\alpha_n,\beta_1,\ldots,\beta_m)^{\operatorname{t}}$.
\end{proof}

As indicated before, we will see later that the restriction $\alpha_1=0$ is not as strong as it may appear: By tensoring a given hypergeometric $\cD$-module with an appropriate Kummer module, we can always reach this assumption.

We will end this section by explaining how the above construction of GKZ-systems and the dimensional reduction to hypergeometric $\cD_{\gmt}$-modules can be understood at the level of $\cR$-modules. Recall (see the introduction) that for a smooth algebraic variety $X$ with local coordinates $(x_1,\ldots,x_n)$ the sheaf $\cRint_{\AA_z^1\times X}$ is the subsheaf of $\cD_{\AA^1_z\times X}$ locally generated by $z^2\partial_z$ and $(z\partial_{x_i})_{i=1,\ldots, n}$.

\begin{defi}
Let $n\geq m$ two positive integers, and let $d=n-m$. Let $\beta\in\CC^d$ be a vector and let $A=(a_{ij})\in\text{M}(d\times n,\ZZ)$ be an integer matrix.
Consider the affine space $\gm^n$ with coordinates $\lambda_1,\ldots,\lambda_n$, and let $\LL \subset \ZZ^n$ be the kernel of the linear map $\ZZ^n\rightarrow \ZZ^d$ given by left multiplication by the matrix $A$, whose elements will be denoted by $\underline{l}=(l_1,\ldots, l_n)$. Then, the GKZ-hypergeometric $\cR$-module is
$$
\wh{\cM}_{A}^{(\beta_0,\beta)}:=\cRint_{\AA_z^1\times \GG_m^n}/\cI,
$$
where $\beta_0\in\CC$ and $\cI$ is generated by
$$\begin{array}{c}
\displaystyle\prod_{j:l_j>0} (z\partial_{\lambda_j})^{l_j}-\prod_{j:l_j<0} (z\partial_{\lambda_j})^{-l_j},\; \underline{l}\in \LL,\\
z^2\partial_z+\lambda_1z\partial_{\lambda_1}+\ldots+\lambda_nz\partial_{\lambda_n}-z\beta_0, \\
\displaystyle\sum_{j=1}^n a_{kj}\lambda_j z\partial_{\lambda_j} - z\beta_k,\; k=1,\ldots,d.
\end{array}
$$
\end{defi}

Note that we can recover the GKZ-hypergeometric $\cD$-module $\cM_{A}^{\beta}$ from Definition \ref{def:GKZ} by restricting $\wh\cM_A^{(\beta_0,\beta)}$ to $z=1$. In the special case of our original matrix from Corollary \ref{coro:hypGKZ} the generators of $\cI$ are
$$\begin{array}{c}
(z\partial_{\lambda_1})\cdot\ldots\cdot(z\partial_{\lambda_n})- (z\partial_{\lambda_{n+1}})\cdot\ldots\cdot(z\partial_{\lambda_{n+m}}),\\
z^2\partial_z+\lambda_1z\partial_{\lambda_1}+\ldots+\lambda_{n+m}z\partial_{\lambda_{n+m}}-z\beta_0, \\
\lambda_1z\partial_{\lambda_1}+\lambda_{n+i}z\partial_{\lambda_{n+i}}-z\beta_i,\; i=1,\ldots,m,\\
\lambda_1z\partial_{\lambda_1}-\lambda_iz\partial_{\lambda_i}+z\alpha_i,\; i=2,\ldots,n.
\end{array}$$

Moreover, we must also consider the corresponding $\cR$-module for hypergeometric $\cD$-modules. Both kinds of $\cR$-modules will play a significant role in the proof of Theorem \ref{thm:HypIrrMHM}.

\begin{defi}\label{def:H_Hut}
Let $(n,m)\neq(0,0)$ be a pair of natural numbers, and let $\alpha_1,\ldots,\alpha_n$ and $\beta_1,\ldots,\beta_m$ be elements of $\CC$. The (classical) hypergeometric $\cR$-module (of type $(n,m)$) associated with the $\alpha_i$ and the $\beta_j$, denoted by $\wh\cH(\alpha_i;\beta_j)$, is defined as the quotient of $\cRint_{\AA_z^1\times\gmt}$ by the left ideal generated by
$$
P=z^2\partial_z+(n-m)tz\partial_t+\gamma z \,\text{ and }\,H=\prod_{i=1}^n z(t\partial_t-\alpha_i)-t\prod_{j=1}^m z(t\partial_t-\beta_j),
$$
where $\gamma=-\sum_{i=1}^n \alpha_i+\sum_{j=1}^m\beta_j$.
\end{defi}

The choice of the operator $P$ may seem odd, but as we will see, it is indeed very natural. In fact, we have the following extension of Corollary \ref{coro:hypGKZ} to the realm of $\cR$-modules.
\begin{lemma}\label{lem:RhypGKZ}
Let $\wh\cH(\alpha_i;\beta_j)$ be a classical hypergeometric $\cR_{\AA_z^1\times\gmt}$-module of type $(n,m)$ with $n>0$ and $\alpha_1=0$. Let $A\in\operatorname{M}((m+n-1)\times(n+m),\ZZ)$, $\beta\in\CC^{n+m-1}$ and $\iota:\gmt\hra\GG_m^{n+m}$ be as in the statement of Corollary \ref{coro:hypGKZ}. Then
$$\wh\cH(\alpha_i;\beta_j)\cong\iota^+\wh\cM_A^{(0,\beta)}.$$
\end{lemma}
\begin{proof}
The inverse image functor in the category of $\cR$-modules is induced by the usual inverse image functor of $\cO$-modules, $(\id_{\AA_z^1}\times\iota)^*$ in this case (cf. \cite[\S2.1.6.2]{Mo13}). Then it is easy to see that
$$
\wh\cH(\alpha_i;\beta_j)\cong\iota^+\wh\cM_A^{(0,\beta)}.
$$
Namely, we replace $z\lambda_i\dd_{\lambda_i}$ by $z\lambda_1\dd_{\lambda_1}-z\alpha_i$ if $i=2,\ldots,n$ or by $-z\lambda_1\dd_{\lambda_1}+z\beta_{i-n}$, if $i=n+1,\ldots,n+m$. Since we can invert $\lambda_i$ in $\cRint_{\AA^1_z\times \GG_m^n}$, we
present $\wh\cM_A^{(0,\beta)}$ as the $\cO_{\gm^{n+m}}$-module $\cO_{\gm^{n+m}}\langle z^2\dd_z,z\lambda_1\dd_{\lambda_1}\rangle/\cJ$, where $\cJ$ is generated by $$\lambda_{n+1}\cdot\ldots\cdot\lambda_{n+m}\prod_{i=1}^nz(\lambda_1\dd_{\lambda_1}-\alpha_i)- (-1)^m\lambda_1\cdot\ldots\cdot\lambda_n\prod_{i=1}^mz(\lambda_1\dd_{\lambda_1}-\beta_i)\,\text{ and }\, z^2\dd_z+(n-m)z\lambda_1\dd_{\lambda_1}+\gamma z.$$
Now the inverse image by $\iota$ amounts simply to set $\lambda_1=t$ and $\lambda_i=1$ for $i=2,\ldots,{n+m}$ in the generators of the ideal, from which the desired isomorphism follows, up to multiplying $t$ by $-1$.
\end{proof}

We can formulate at this point one of the main results of this paper. Its full proof will occupy the entire next section.
\begin{thm}\label{thm:HypIrrMHM}
Let $\alpha_1,\ldots,\alpha_n$ be real numbers, belonging to the interval $[0,1)$ and increasingly ordered, and let $\gamma=-\sum_{i=1}^n \alpha_i$. Then the $\cR_{\AA^1_z\times\gmt}$-module $\wh\cH:=\wh\cH(\alpha_i,\emptyset)=\cRint_{\AA^1_z\times \gmt}/(P,H)$, where
$$
P=z^2\partial_z+ntz\partial_t+\gamma z \,\text{ and }\,H=\prod_{i=1}^n z(t\partial_t-\alpha_i)-t,
$$
underlies an irregular Hodge module, i.e., an object of $\IrrMHM(\gmt)$. It is the unique
irregular Hodge module whose associated $\cD_{\gmt}$-module is $\cH(\alpha_i;\emptyset)$. Moreover, $\wh\cH$ can be extended in a unique way to an $\cRint_{\AA^1_z\times \PP^1}$-module, $\wh\cH_{pr}$, such that it underlies an object of $\IrrMHM\left(\PP^1\right)$.
\end{thm}
\begin{proof}[Proof of unicity]
We will give here a proof of the two unicity statements in the above theorem, postponing the proof of the main statement to page \pageref{page:MainProof} below.

Consider any twistor $\cD$-module $\wh\cH'$ on $\gmt$ whose underlying $\cD_{\GG_{m,t}}$-module is $\cH$. Since the functor $\Xi_{\text{DR}}$ is faithful by \cite[Rem. 7.2.9]{Mo13}, we have an injection of Hom groups
$$\Hom_{\MTM(\GG_{m,t})}(\wh\cH,\wh\cH')\hra\Hom_{\cD_{\GG_{m,t}}}(\cH,\cH).$$
But $\cH$ is irreducible, so its only endomorphism is the identity and then, a twistor $\cD$-module underlying $\cH$ is unique, if it exists.

On the other hand, let $j:\GG_{m,t}\hra\PP^1$ be the canonical inclusion and consider the $\cD_{\PP^1}$-module $\cH_{pr}:=j_{\dag+}\cH$. It is an irreducible holonomic $\cD_{\PP^1}$-module, because so is $\cH$ by Proposition \ref{prop:IrredHyp}. Then it gives rise to a unique pure integrable twistor $\cD$-module $\wh\cH_{pr}$ on $\PP^1$ by \cite[Thm. 1.4.4]{Mo5} and \cite[Rem. 1.40]{Sa15}. In addition, its underlying $\cD_{\PP^1}$-module $\cH_{pr}$ is rigid by virtue of Proposition \ref{prop:RigidHyp}. As a consequence, we can invoke \cite[Thm. 0.7]{Sa15} and claim that such twistor $\cD$-module on $\PP^1$ is in fact an object of $\IrrMHM(\PP^1)$. Take now $\wh\cH':=j^+\wh\cH_{pr}$, which is an irregular mixed Hodge module whose underlying $\cD_{\GG_{m,t}}$-module is $\cH$, by \cite[Prop. 14.1.24]{Mo13}. The we must have, as was just shown, $\wh{\cH}'\cong \wh{\cH}$, so that the extension $\wh{\cH}_{pr}$ of $\wh{\cH}$ is unique, as claimed.
\end{proof}

\label{page:Roadmap}
The main point in the above theorem is that $\wh{\cH}$ underlies an irregular mixed Hodge module. Since the proof of this fact is rather long, and will be carried out in the next section through various intermediate results, we would like to orient the reader by giving here an overview
of these steps. We will restrict the sketch to the case where $\alpha_1=0$, this is also the first (and main) step in the
actual proof below on page \pageref{page:MainProof}. The general case can be rather easily deduced from this special one by considering Kummer $\cD$- resp. $\cR$-modules.

The first point is to realise $\wh{\cH}$ in a geometric way. For this purpose, consider the following two
families of Laurent polynomials
$$
f(y_1,\ldots,y_{n-1},\lambda_1,\ldots,\lambda_n):= -\lambda_1 \cdot y_1\cdot\ldots\cdot y_{n-1}-\frac{\lambda_2}{y_1}-\ldots-\frac{\lambda_n}{y_{n-1}}
$$
and
$$
{\,'\!}f(y_1,\ldots,y_{n-1},t):=-t\cdot y_1\cdot\ldots\cdot y_{n-1}-\frac{1}{y_1}-\ldots-\frac{1}{y_{n-1}}.
$$
where $y_k,\lambda_i, t\in \GG_m$. Write $\iota: \gmt\hookrightarrow \GG_m^n$, $t\mapsto (t,1,\ldots,1)$, so that we
have the cartesian diagram
$$
\begin{tikzcd}
\GG_m^{n-1}\times \gmt \ar{rr} \ar[swap]{dd}{{\,'\!}f\times \pi_2} & & \GG_m^{n-1}\times \GG_m^n \ar{dd}{f\times \pi_2} \\ \\
\AA^1_{\lambda_0}\times \gmt \ar{rr}{\id_{\AA^1_{\lambda_0}}\times \iota} && \AA^1_{\lambda_0}\times \GG_m^n
\end{tikzcd}
$$
where we denote the coordinate on the affine line corresponding to the value of $f$ resp. of ${\,'\!}f$ by $\lambda_0$,
for reasons that will become clear later.

We consider the so-called twisted cohomology groups associated to the morphisms $f$ and ${\,'\!}f$. It can be shown (see Proposition \ref{prop:TwdeRham} below) that
$$
\wh{\cM}^{(0,\alpha)}_A\cong\cH^{n-1}\left(\pi_{2,*}\Omega_{\GG_m^{n-1}\times \GG_m^n/\GG_m^n}^{\bullet+d}[z],z\left(d-\kappa(\alpha)\wedge\right)-df\wedge\right),
$$
where $(0,\alpha)=(0,\alpha_1,\alpha_2,\ldots,\alpha_n)=(0,0,\alpha_2,\ldots,\alpha_n)$,
where $\kappa(\alpha)=\sum_{j=1}^{n-1} \alpha_{j+1} dy_j/y_j$ and $A$ is the matrix from Corollary \ref{coro:hypGKZ} for the
case $m=0$.
Moreover, since the twisted cohomology groups involve complexes of relative differential forms, we have
$$
\begin{array}{c}
\displaystyle \iota^+ \cH^{n-1}\left(\pi_{2,*}\Omega_{\GG_m^{n-1}\times \GG_m^n/\GG_m^n}^{\bullet+d}[z],z\left(d-\kappa(\alpha)\wedge\right)-df\wedge\right) \\ \\
\cong
\displaystyle \cH^{n-1}\left(\pi_{2,*}\Omega_{\GG_m^{n-1}\times \gmt/\gmt}^{\bullet}[z],z\left(d-\kappa(\alpha)\wedge\right)-d{\,'\!}f\wedge\right),
\end{array}
$$
so that by using Lemma \ref{lem:RhypGKZ} we obtain an isomorphism of $\cRint_{\AA^1_z\times\gmt}$-modules
\begin{equation}\label{eq:IdentRMod}
\wh{\cH}
\cong
\cH^{n-1}\left(\pi_{2,*}\Omega_{\GG_m^{n-1}\times \gmt/\gmt}^{\bullet}[z],z\left(d-\kappa(\alpha)\wedge\right)-d{\,'\!}f\wedge\right).
\end{equation}
As a second step, we will realize the right hand side of the above isomorphism in a different way. Namely, write
${\,'\!}\varphi:=({\,'\!}f,\pi_2):\GG_m^{n-1}\times\gmt\rightarrow\AA^1_{\lambda_0}\times\gmt$ and consider the direct image complex
${\,'\!}\varphi_+\cO_{\GG_m^{n-1}}^\alpha\in \Dbrh(\cD_{\AA^1_{\lambda_0}\times\gmt})$, where
$\cO_{\GG_m^{n-1}}^\alpha=\cD_{\GG_m^{n-1}}/(y_j\partial_{y_j}+\alpha_{j+1}+1)_{j=1,\ldots,n-1}$. We are interested
in the top cohomology of this complex. Standard techniques for the calculation
of direct images of $\cD$-modules show that it is given by
$$
M:=\frac{\pi_{2,*}\Omega_{\GG_m^{n-1}\times \gmt/\gmt}^{n-1}[\partial_{\lambda_0}]}{(d-\partial_{\lambda_0} \cdot d{\,'\!}f\wedge)\pi_{2,*}\Omega_{\GG_m^{n-1}\times \gmt/\gmt}^{n-2}[\partial_{\lambda_0}]}
$$
We will use a variant of the Fourier-Laplace transformation (called localized partial Fourier-Laplace transformation,
and denoted by $\FL^{\loc}_{\gmt}$,
see Definition \ref{def:FL} and Definition \ref{def:plocFL} below) exchanging the operator $\partial_{\lambda_0} \cdot $
into $z^{-1} \cdot $, the operator $\lambda_0 \cdot$ into $z^2\partial_z \cdot $, and localizing along $z=\infty$. Then we have
an isomorphism of $\cD_{\AA^1_z\times \gmt}$-modules
$$
\FL^{\loc}_{\gmt}(M) =
\frac{\pi_{2,*}\Omega_{\GG_m^{n-1}\times \gmt/\gmt}^{n-1}[z^\pm]}{(z\cdot d- d{\,'\!}f\wedge)\pi_{2,*}\Omega_{\GG_m^{n-1}\times \gmt/\gmt}^{n-2}[z^\pm]}
\cong
\wh{\cH}[z^\pm] \supset \wh{\cH}.
$$
One of the main points of the proof in the next section is to give a good description of the image of $\wh{\cH}$
inside $\FL^{\loc}_{\gmt}(M)$ under this isomorphism. One such description is given by the isomorphism of $\cRint_{\AA^1_z\times \gmt}$-modules in the displayed formula \eqref{eq:IdentRMod}. However, we cannot, a priori, obtain any Hodge theoretic
information on $\wh{\cH}$ from \eqref{eq:IdentRMod}. On the other hand, we know that $M$ underlies an algebraic complex mixed Hodge module on $\AA^1_{\lambda_0 \times \gmt}$ (since it is the direct image of such an object on $\GG_m^{n-1}$), and hence it comes equipped with a certain good filtration $F^H_\bullet M$ (the Hodge filtration).
There is a general procedure, explained below in Definition \ref{def:G0F} and Lemma \ref{lem:MHMIntoIrrMHMbyFL}, which constructs,
given a filtered $\cD_{\AA_{\lambda_0}\times \gmt}$-module $(N,F_\bullet)$,
a $\cRint_{\AA^1_z\times \gmt}$-module called $G_0^F \FL^{\loc}_{\gmt} N$ such that its localisation
$G_0^F \FL^{\loc}_{\gmt} N \otimes_{\cO_{\AA^1_z\times \gmt}} \cO_{\AA^1_z\times \gmt}[z^{-1}]$ equals
$\FL^{\loc}_{\gmt} N$. Then we show in Theorem \ref{thm:EqualFourierFilt} that
$$
\wh{\cH} \cong G_0^{F^H} \FL^{\loc}_{\gmt} M
$$
up to a shift of the Hodge filtration. Actually, the proof is not that direct, since we have
to identify $\FL^{\loc}_{\gmt}(M)$ with the localized partial Fourier-Laplace transformation of some other $\cD_{\AA^1_{\lambda_0}\times \gmt}$-module (called $M_{\dag+}$), which underlies
a pure polarizable Hodge module. It is constructed by taking a compactification
of ${\,'\!}f$, i.e., a projective morphism defined on a quasi-projective (usually singular)
variety constructed from the toric compactification of $\GG_m^{n-1}$ inside $\PP^n$.
Then $M_{\dag+}$ is obtained as the direct image under this projective morphism of
a certain intersection cohomology module. Now it is known (see \cite[Cor. 0.5]{Sa15}) that if
$M_{\dag+}$ underlies a pure polarizable Hodge module, the $\cRint_{\AA^1_z\times \gmt}$-module
$G_0^{F^H}\FL^{\loc}_{\gmt}M_{\dag+}$ underlies an irregular Hodge module,
which finishes the proof. Notice that the proof of the identification $\FL^{\loc}_{\gmt}(M)\cong
\FL^{\loc}_{\gmt}(M_{\dag+})$ is derived from a similar isomorphism for the direct images of the
morphisms $\varphi=(f,\id_{\GG_m^n})$ resp. its compactification, rather than for ${\,'\!}\varphi$, and is done via
the formalism of Radon transformations for regular holonomic $\cD$-modules (in the same way as in \cite{Reich2}
and \cite{ReiSe2, ReiSe3}).

\section{Hodge modules and Fourier-Laplace transformation}

Let $\alpha_1,\ldots,\alpha_n$ be real numbers, and consider the hypergeometric $\cD_{\gm}$-module $\cH=\cH(\alpha_i;\emptyset)$.
As we have mentioned before, the goal of this section is to prove Theorem \ref{thm:HypIrrMHM} above, showing that the $\cR$-module $\wh\cH:=\wh\cH(\alpha_i;\emptyset)$ from Definition \ref{def:H_Hut} underlies an object of $\IrrMHM(\gm)$ (the abelian category of exponential Hodge modules as defined in \cite{Sa15}, see the introduction). We have already indicated several times
that we will use the GKZ-hypergeometric $\cR$-module $\wh\cM_A^{(0,\alpha)}$ for the matrix $A$ from corollary \ref{coro:hypGKZ},
but where $m=0$. However, we will start with a more general situation, and specify the assumptions we need when going on with the proof.

Let $d<n$ be two positive integers, and take a parameter vector $\beta\in\CC^d$. For the Hodge theoretic questions we are interested in, only real parameter vectors are relevant, but we will work with this
more general setting until Remark \ref{rem:realExponents} below.
Let $A=(\underline{a}_1,\ldots,\underline{a}_n)\in \text{M}(d\times n,\ZZ)$ be an integer matrix satisfying the following:

\begin{assump}\label{assump:AssMatrix}
$\mbox{}$
\begin{enumerate}[(i)]
\item $\ZZ A =\ZZ^d$, here $\ZZ A:=\sum_{i=1}^n \ZZ \underline{a}_i$,
\item Let $\Delta:=\Conv(\underline{a}_1,\ldots,\underline{a}_n)$
be the convex hull in $\RR^d$ of the vectors given by the columns of the matrix $A$. Then for any proper face $\Gamma\subset \Delta$ we require
that the set $\{\underline{a}_i \,|\, \underline{a}_i\in \Gamma\}$ is part of a $\QQ$-basis of $\QQ^d$.
\item The origin lies in the interior of $\Delta$.
\end{enumerate}
\end{assump}

\begin{rem}\label{rem:AssMatrix}
\begin{enumerate}[(i)]
\item These assumptions are in particular satisfied if $\underline{a}_1,\ldots,\underline{a}_n$ are the primitive integral generators
of the rays of the fan $\Sigma$ defining a toric Fano orbifold $X_\Sigma$, since in this case it is known (see \cite[Lem. 3.2.1 and \S~3.5]{CK}) that
$\Sigma$ is the union of the cones over the proper faces of $\Delta$, so that assumption (ii) is satisfied by the fact that the cones of $\Sigma$ are simplicial. Moreover, for Fano toric varieties the origin is the \emph{only} integer point in the interior of $\Delta$, so (iii) obviously holds. In any case, we see that these conditions are satisfied
for the example
\begin{equation}\label{eq:matAProjSpace}
A=
\begin{pmatrix}
1 & -1 & 0 & 0 &\ldots  & 0  \\
1 &  0 & -1 & 0 & \ldots & 0 \\
\vdots &\ldots \\
1 &  0 & 0 & 0 & \ldots & -1
\end{pmatrix}
\in \text{M}\left((n-1) \times n, \ZZ\right),
\end{equation}
since these are the generators of the rays of the fan of $\PP^{n-1}$. As we have seen in Corollary \ref{coro:hypGKZ}, this is the matrix
we have to look at when we want to express the classical hypergeometric $\cD$-module of type $(n,0)$ and where $\alpha_0=0$ as an inverse image of the GKZ-system $\cM_A^\beta$
where $\beta=\alpha:=(\alpha_2,\ldots, \alpha_n)^{\text{t}}\in \RR^{n-1}$.
\item
Assumption (iii) from above implies in particular that there is a relation $\underline{l}=(l_1,\ldots,l_n) \in \textup{ker}(A)\subset \ZZ^n$ such
that $l_i>0$ for $i=1,\ldots, n$. It follows then from assumption (i) that the semi-group $\NN A=\sum_{i=1}^{n} \NN \underline{a}_i$ equals $\ZZ^d$ (since for any $\underline{c}\in \ZZ^d$, a linear combination
$\underline{c}=\sum_{i=1}^{n} n_i \underline{a}_i$ with integer coefficients $n_i$ can be turned into a combination
with positive coefficents by adding the vector $\underline{0}=\sum_{i=1}^{n} l_i \underline{a}_i$ sufficiently many times). This fact will be used later (see the proof of Proposition \ref{prop:TwdeRham}).
\end{enumerate}
\end{rem}

Let $S_1=\gm^d=\Spec(\CC[y_1^\pm,\ldots,y_d^\pm])$ and $S_2=\gm^n=\Spec(\CC[\lambda_1^\pm,\ldots,\lambda_n^\pm])$ be two algebraic tori, and consider the affine space $V=\AA^{n+1}$ with coordinates $\lambda_0,\lambda_1,\ldots,\lambda_n$. Let $V^\vee$ be the dual space with coordinates $w_0,w_1,\ldots,w_n$, and we also set $\tau:=-w_0$ and $z:=\tau^{-1}$. We decompose $V=\AA^1_{\lambda_0}\times W$, and consider $S_2\subset W$ as an open subset.

Consider the following family of Laurent polynomials
\begin{equation}\label{eq:TwdeRham}
\begin{array}{rcl}
\varphi: S_1 \times S_2 & \longrightarrow & \AA^1_{\lambda_0}\times S_2 \\
\left(\uy,\ulambda\right) & \longmapsto & \displaystyle \left(-\sum_{i=1}^n \lambda_i \uy^{\underline{a}_i},\lambda_1,\ldots,\lambda_n\right)
\end{array},
\end{equation}
which in the case of the matrix from equation \eqref{eq:matAProjSpace} becomes
\begin{equation}\label{eq:TwdeRham-concrete}
\varphi(y_1,\ldots,y_{n-1},\lambda_1,\ldots,\lambda_n)=\left(-\lambda_1 \cdot y_1\cdot\ldots\cdot y_{n-1}-\frac{\lambda_2}{y_1}-\ldots-\frac{\lambda_n}{y_{n-1}},\lambda_1,\ldots,\lambda_n\right).
\end{equation}
Write $f:S_1\times S_2\rightarrow \AA_{\lambda_0}^1$ for the composition of $\varphi$ with the first projection $\AA_{\lambda_0}^1\times S_2\rightarrow\AA^1_{\lambda_0}$. Similarly, for any fixed $\underline{\lambda}\in S_2$,
we write $f_{\underline{\lambda}}:=f_{|S_1\times\{\underline{\lambda}\}}:S_1\rightarrow \AA^1_{\lambda_0}$ for the restriction of $f$ to the parameter value $\underline{\lambda}$.
Let us first quote the following statement from \cite[Lemma 2.8.]{ReiSe}. Since the input data in loc. cit. are fans
of toric varieties, we will copy the proof here to make it fit the assumptions above. Recall (see \cite{Kouch}) that
a Laurent polynomial $f_{\underline{\lambda}}=\sum_{i=1}^{n} \lambda_i \underline{y}^{\underline{a}_i}$ is called
convenient if $0$ lies in the interior of $\Delta$ and non-degenerate if
for all proper faces $\delta\subset \Conv(0,\underline{a}_1,\ldots,\underline{a}_n)$ not containing the origin, the Laurent polynomial $f^\delta_{\underline{\lambda}} =
\sum_{i:\underline{a}_i\in \delta} \lambda_i \underline{y}^{\underline{a}_i}$ has no critical points in $S_1$.
Notice that for matrices $A$ satisfying assumption (iii) from above, this last condition is equivalent to asking that $f_{\underline{\lambda}}^\delta$ is non-singular for
all proper faces $\delta\subset \Delta$.
\begin{lemma}\label{lem:NonDegen}
The Laurent polynomial $f_{\underline{\lambda}}:S_1\rightarrow \AA^1_z$ is non-degenerate and convenient for any $\underline{\lambda}\in S_2$.
\end{lemma}
\begin{proof}
Obviously $f_{\underline{\lambda}}$ is convenient by assumption (iii) from above.
Let $\delta\subset \Delta$ be a face of codimension $d+1-l$, with $l=1,\ldots, d$.
Let $\{i_1,\ldots,i_l\}\subset \{1,\ldots,n\}$ such that $\{\underline{a}_i\in\delta\}=\{\underline{a}_{i_1},\ldots, \underline{a}_{i_l}\}$,
notice that because of assumption (ii), we cannot have more than $l$ vectors in a face of dimension $l-1$.
Since the vectors
$\underline{a}_{i_1}, \ldots, \underline{a}_{i_l}$ are linearly independent over $\QQ$,
the matrix
$$
C:=\begin{pmatrix}
  a_{i_1 1} & \ldots & a_{i_l 1} \\
  \vdots & \vdots & \vdots \\
  a_{i_1 d} & \ldots & a_{i_l d} \\
\end{pmatrix}
$$
has full rank (equal to $l$) and hence the system
$$
C\cdot
\begin{pmatrix}
  \lambda_{i_1} \underline{y}^{\underline{a}_{i_1}} \\
  \vdots \\
  \lambda_{i_l} \underline{y}^{\underline{a}_{i_l}}
\end{pmatrix} = 0,
$$
which is the system of critical point equations $(y_k\partial_{y_k} f^\delta_{\underline{\lambda}} =0)_{k=1,\ldots,d}$,
has no nontrivial solution; the trivial one $\lambda_i \cdot \underline{y}^{\underline{a}_i}=0$ for all $i\in \{i_1,\ldots,i_l\}$ is not valid since $\underline{\lambda}\in S_2$ and we are looking for solutions $\underline{y}\in S^1$. Hence
$f_{\underline{\lambda}}^\delta$ is non-singular on $S_1$, and so $f_{\underline{\lambda}}$ is non-degenerate.
\end{proof}
From this we can deduce
the following statement, which is a variant of
\cite[Lem. 2.13]{ReiSe}. However,  we will give the proof here for the convenience of the reader.
\begin{lemma}
$\wh{\cM}_{A}^{(\beta_0,\beta)}$ is locally $\cO_{\AA^1_z\times S_2}$-free of rank $n!\cdot\operatorname{vol}(\Delta)$ (considering the normalized volume in $\RR^n$ such that $[0,1]^n$ has volume one). For the case of the matrix in equation \eqref{eq:matAProjSpace}, this rank
equals $n$.
\end{lemma}
\begin{proof}
There is an isomorphism of $\cO_{\AA^1_z\times S_2}$-modules (or even of $\cR_{\AA^1_z\times S_2}$-modules)
$$
\wh{\cM}_{A}^{(\beta_0,\beta)}\cong\frac{\cR_{\AA^1_z\times S_2}}{
\left(\prod_{j:l_j>0} (z\partial_{\lambda_j})^{l_j}-\prod_{j:l_j<0} (z\partial_{\lambda_j})^{-l_j}\right)_{l\in\LL}
+\left(\sum_{j=1}^n a_{kj}\lambda_j z\partial_{\lambda_j} - z\beta_k\right)_{k=1,\ldots,d}
},
$$
so it suffices to prove the statement for the module on the right hand side of this equation.
We consider the filtration induced on it by the filtration
on $\cR_{\AA^1_z\times S_2}$ for which $z\partial_{\lambda_i}$ has degree $1$ and any element of $\cO_{\AA^1_z\times S_2}$ has degree
zero. The graded module with respect to this filtration is a sheaf on $\AA^1_z\times T^* S_2$, and we first need to show
that its support lies in the zero section, i.e., in the subspace $\AA^1_z\times S_2$. Notice that the symbols
of the operators in the ideal
$\left(\prod_{j:l_j>0} (z\partial_{\lambda_j})^{l_j}-\prod_{j:l_j<0} (z\partial_{\lambda_j})^{-l_j}\right)_{l\in\LL}
+\left(\sum_{j=1}^n a_{kj}\lambda_j z\partial_{\lambda_j} - z\beta_k\right)_{k=1,\ldots,d}$ with respect to the
filtration of $\cR_{\AA^1_z\times S_2}$ defined above are the same as the symbols
of the operators of the usual hypergeometric ideal
$\left(\prod_{j:l_j>0} \partial_{\lambda_j}^{l_j}-\prod_{j:l_j<0} \partial_{\lambda_j}^{-l_j}\right)_{l\in\LL}
+\left(\sum_{j=1}^n a_{kj}\lambda_j \partial_{\lambda_j} - \beta_k\right)_{k=1,\ldots,d}$ with respect to the
(usual) order filtration on $\cD_{S_2}$. Hence by the arguments of \cite[Lem. 3.1 to Lem. 3.3]{Adolphson},
we obtain that the variety cut out by the symbols of the operators in
$\left(\prod_{j:l_j>0} (z\partial_{\lambda_j})^{l_j}-\prod_{j:l_j<0} (z\partial_{\lambda_j})^{-l_j}\right)_{l\in\LL}
+\left(\sum_{j=1}^n a_{kj}\lambda_j z\partial_{\lambda_j} - z\beta_k\right)_{k=1,\ldots,d}$ is $\AA^1_z\times S_2$,
notice that here the fact that $f_{\underline{\lambda}}$ is non-degenerate for all $\underline{\lambda}\in S_2$
(i.e., the statement of the last lemma) plays a crucial role.

From $\textup{supp}(\textup{gr}(\wh{\cM}^{(\beta_0,\beta)}_A))\subseteq \AA^1_z\times S_2$ one deduces as in ordinary $\cD$-module theory that $\wh{\cM}_{A}^{(\beta_0,\beta)}$ is
$\cO_{\AA^1_z\times S_2}$-coherent.
Now the restriction $\wh{\cM}_A^{(0,\beta)}/z\cdot \wh{\cM}_A^{(0,\beta)}$ is isomorphic
to the Jacobian algebra $\textup{Jac}(f)=\cO_{S_1\times S_2}/(\partial_{y_1} f, \ldots, \partial_{y_d} f)$
(see \cite[Lem. 2.12]{ReiSe}), and that the latter has rank equal to $n!\cdot \textup{vol}(\Delta)$ (see \cite[Thm. 1.16]{Kouch}).
Moreover, the localized object $\wh{\cM}_A^{(0,\beta)}\otimes_{\cO_{\AA^1_z\times S_2}} \cO_{\AA^1_z\times S_2}[z^{-1}]$
has a $\cD_{\AA^1_z\times S_2}[z^{-1}]$-module structure, and is $\cO_{\AA^1_z\times S_2}[z^{-1}]$-coherent, hence
$\cO_{\AA^1_z\times S_2}[z^{-1}]$-free. Its rank can also be calculated as the holonomic rank of ordinary GKZ-systems,
see \cite[Prop. 2.7 (3)]{ReiSe}, and equals $n!\cdot\operatorname{vol}(\Delta)$.

This shows that the module $\wh{\cM}^{(\beta_0,\beta)}_A$ itself is $\cO_{\AA^1_z\times S_2}$-free of the same rank
$n!\cdot\operatorname{vol}(\Delta)$.
\end{proof}

For any complex number $\beta$, recall that the Kummer $\cD$-module of parameter $\beta$ is by defintion the quotient $\cK_\beta:=\cD_{\gmt}/(t\dd_t-\beta)$. We will also use in this section $\cR$-modules arising from such $\cD$-modules. Here is the precise definition.

\begin{defi}
For any complex number $\beta$, we define the Kummer $\cR$-module of parameter $\beta$ as the cyclic $\cRint_{\AA^1_z\times \gmt}$-module
$$\wh\cK_{\beta}:=\cRint_{\AA^1\times \gmt}/(z^2\dd_z,zt\dd_t-z\beta).$$
\end{defi}

\begin{rem}\label{rem:KummerHodge}
Although both kinds of Kummer modules can be defined for any complex value of their parameters, in the end we will be interested only in the real case to make use of their Hodge properties. Indeed, since both have no singularities at $\gmt$, if $\cK_\beta$ were a complex Hodge module it would be in fact a complex variation of Hodge structures, and by (the first part of) the proof of \cite[Lem. 4.5]{Sch}, $\beta$ ought to be real.

On the other hand, $\wh\cK_\beta$ is clearly the Rees module of $\cK_\beta$ together with the trivial filtration $F_\bullet$ such that $F_k=0$ for $k<0$ and $F_k=\cK_\beta$ for $k\geq0$. As described in \cite[Prop. 13.5.4]{Mo13} and \cite[Thm. 0.2]{Sa15}, it gives rise to an integrable pure twistor $\cD$-module on $\gmt$, which belongs as well to $\IrrMHM(\gm)$. It is also described as a harmonic bundle at \cite[\S2.1.9]{Mo13}.
\end{rem}

\begin{defi}\label{def:Oalpha}
For any smooth complex algebraic variety $X$, and for any $\beta\in\CC^d$, we will denote by $\cO_{S_1\times X }^\beta$ the $\cD_{S_1\times X }$-module corresponding to the structure sheaf of $S_1 \times X$ twisted by $\uy^{-1-\beta}$, that is,
$$\cO_{S_1\times X }^\beta:=\frac{\cD_{S_1}}{\left(y_k\dd_{y_k}+\beta_k+1\,:\, k=1,\ldots,d\right)}\boxtimes\cO_ X =:\cO_{S_1}^\beta\boxtimes\cO_ X .$$
\end{defi}

Note that for any other $\beta'\in\CC^d$ such that $\beta-\beta'\in\ZZ^d$, $\cO_{S_1\times X }^\beta\cong\cO_{S_1\times X }^{\beta'}$. These modules underly complex Hodge modules if and only if the components of the parameter vector are real numbers, since they are the corresponding exterior products of the Kummer modules $\cK_{-\beta_1},\ldots,\cK_{-\beta_d}$ and $\cO_ X $.

With these notations, the following result is a special case of
\cite[Prop. 3.21]{ReiSe2}, taking into account the twist of $\cO_{S_1\times S_2}$ by $\uy^{-1-\beta}$. However, we prefer
to give a direct proof here, which is a simplified variant of the corresponding statement in \cite[Prop. E.6]{Mo15}.
Notice that the idea of this approach goes back to the so-called ``better behaved GKZ-systems'' of Borisov-Horja (see \cite{BorHor}).
\begin{prop}\label{prop:TwdeRham}
There exists an isomorphism of $\cRint_{\AA^1_z\times S_2}$-modules
$$
\cH^0\left(\pi_{2,*}\Omega_{S_1\times S_2/S_2}^{\bullet+d}[z],z\left(d-\kappa(\beta)\wedge\right)-df\wedge\right) \longrightarrow \wh{\cM}_A^{(0,\beta)}
$$
for any $\beta\in\CC^d$, given by sending $\omega:=\prod_{j=1}^d dy_j/y_j$ to $[1]\in\wh{\cM}_A^{(0,\beta)}$, where
we write $\kappa(\beta)$ for $\sum_{j=1}^d \beta_j dy_j/y_j$ and $\pi_2$ for the the second canonical projection $S_1\times S_2\ra S_2$.
\end{prop}
\begin{proof}
We will construct a third module $\cQ/\cK$, and show that both the module $\wh{\cM}_A^{(0,\beta)}$ and the module $\cH^0\left(\pi_{2,*}\Omega_{S_1\times S_2/S_2}^{\bullet+d}[z],z\left(d-\kappa(\beta)\wedge\right)-df\wedge\right)$ are isomorphic to it.
For this purpose, define the free $\cO_{\AA^1_z\times S_2}$-module
$$
\cQ:=\bigoplus_{\underline{c}\in \ZZ^d} \cO_{\AA^1_z\times S_2} \cdot e(\underline{c}),
$$
where $e(\alpha)$ is a symbol representing a generator of $\cQ$. We put a $\cRint_{\AA^1_z\times S_2}$-module structure on $\cQ$ by letting
$$
\begin{array}{rcl}
z \partial_{\lambda_i} e(\underline{c}) & := & e(\underline{c}+\underline{a}_i), \\ \\
z^2\partial_z e(\underline{c}) &:= & \displaystyle  -\sum_{i=1}^{n} \lambda_i \cdot z \partial_{\lambda_i} e(\underline{c}).
\end{array}
$$
Since we have $\NN A = \ZZ^d$ (see Remark \ref{rem:AssMatrix}, (ii)), we conclude from the first of these equations that $\cQ$ is a cyclic $\cRint_{\AA^1_z\times S_2}$-module with generator $e(\underline{0})$.
Moreover, we consider the $\cRint_{\AA^1_z\times S_2}$-submodule $\cK$ of $\cQ$ generated by
$$
\left(\displaystyle \sum_{i=1}^{n} a_{ki} \lambda_i\cdot z\partial_{\lambda_i} + z(c_k-\beta_k)\right)\cdot e(\underline{c})\quad\quad\forall\underline{c}\in\ZZ^d, \;\forall k=1,\ldots,d.
$$
Then we claim that there is an $\cRint_{\AA^1_z\times S_2}$-isomorphism
$$
\phi:\wh{\cM}^{(0,\beta)}_A \stackrel{\cong}{\longrightarrow} \cQ/\cK.
$$
sending $[1]$ to $e(\underline{0})$.
It is clear that $\phi$ is well defined since all operators occurring in the denominator in
the definition of $\wh{\cM}^{(0,\beta)}_A$ act by zero in $\cQ/\cK$: The operators  $\prod_{j:l_j>0} (z\partial_{\lambda_j})^{l_j}-\prod_{j:l_j<0} (z\partial_{\lambda_j})^{-l_j}$ act by zero
already on $\cQ$ (as $\underline{l}\in \LL$);
similarly, $z^2\partial_z+\lambda_1z\partial_{\lambda_1}+\ldots+\lambda_nz\partial_{\lambda_n}$ acts by zero on $\cQ$ because of the
definition of the action of $z^2\partial_z$, and
$\sum_{j=1}^n a_{kj}\lambda_j z\partial_{\lambda_j} - z\beta_k,\; k=1,\ldots,d$ is a generator of $\cK$, so its class is obviously zero
in the quotient $\cQ/\cK$.

An inverse map $\cQ/\cK \rightarrow \wh{\cM}_A^{(0,\beta)}$ is defined as follows: For any $\underline{c} \in \NN^d$, choose a representation
$\underline{c}=\sum_{i=1}^{n} n_i \underline{a}_i$, with $n_i\in\NN$. Then we map $e(\underline{c})$ to the class
of $\prod_{i=1}^{n} (z\partial_{\lambda_i})^{n_i}$ in $\wh{\cM}_A^{(0,\beta)}$. This is well defined: If we take another representation $\underline{c}=\sum_{i=1}^{n} n'_i \underline{a}_i$, then $(n_i-n'_i)_{i=1,\ldots,n}\in\ZZ^n$ lies in $\LL$,
showing that
$$
\displaystyle \prod_{i:n_i>n'_i} (z\partial_{\lambda_i})^{n_i-n'_i}-\prod_{i:n'_i>n_i} (z\partial_{\lambda_i})^{n'_i-n_i} = 0 \in \wh{\cM}^{(0,\beta)}_A,
$$
from what it follows that
$$
\displaystyle \prod_{i=1}^n   (z\partial_{\lambda_i})^{n_i}=\prod_{i=1}^n   (z\partial_{\lambda_i})^{n'_i}\in \wh{\cM}^{(0,\beta)}_A.
$$
Moreover, elements of $\cK$ obviously go to zero in $\wh{\cM}^{(0,\beta)}_A$. It is also clear that this map is inverse to $\phi$.

In order to identify $\cQ/\cK$ with the twisted de Rham cohomology module, notice that the complex
$\left(\pi_{2,*}\Omega_{S_1\times S_2/S_2}^{\bullet+d}[z],z\left(d-\kappa(\beta)\wedge\right)-df\wedge\right)$ is concentrated
in degrees $-d$ to $0$, so we need to compute its top cohomology, i.e., the quotient
\begin{equation}\label{eq:FLGM-system}
\frac{\pi_{2,*}\Omega_{S_1\times S_2/S_2}^{d}[z]}{\left(z\left(d-\kappa(\beta)\wedge\right)-df\wedge\right)\pi_{2,*}\Omega_{S_1\times S_2/S_2}^{d-1}[z]}
\end{equation}
Since we have $\CC[\ZZ^d]\cong H^0(S_1, \cO_{S_1})$, there is an obvious identification of $\cO_{\AA^1_z\times S_2}$-modules,
$$
\cQ \cong \pi_{2,*}\Omega^d_{S_1\times S_2/S_2}[z]
$$
sending $e(\underline{c})$ to $\underline{y}^{\underline{c}}\cdot \omega$.
Now it is easy to see (cf. the proofs of \cite[Prop. E.4 and E.6]{Mo15}) that the image of
$$
\left(z\left(d-\kappa(\beta)\wedge\right)-df\wedge\right)\pi_{2,*}\Omega_{S_1\times S_2/S_2}^{d-1}[z]
$$
can be
identified under this isomorphism with the submodule $\cK$ of $\cQ$. Moreover, one checks that the $\cRint_{\AA^1_z\times S_2}$-module
structures on the quotient \eqref{eq:FLGM-system} and the module $\cQ/\cK$ are compatible, e.g., we have
$z\partial_{\lambda_i} \omega = -(\partial_{\lambda_i} f) \omega=\underline{y}^{\underline{a}_i}\omega$, accordingly with $(z\partial_{\lambda_i})\cdot e(\underline{0})
=e(\underline{a}_i)$.

Combining the two isomorphisms $\cQ/\cK \cong
\cH^0\left(\pi_{2,*}\Omega_{S_1\times S_2/S_2}^{\bullet+d}[z],z\left(d-\kappa(\beta)\wedge\right)-df\wedge\right)$ and
$\wh{\cM}_A^{(0,\beta)} \cong \cQ/\cK$ we obtain the desired one
$$
\cH^0\left(\pi_{2,*}\Omega_{S_1\times S_2/S_2}^{\bullet+d}[z],z\left(d-\kappa(\beta)\wedge\right)-df\wedge\right) \longrightarrow \wh{\cM}_A^{(0,\beta)}
$$
which sends $\omega$ to $[1]$.
\end{proof}

In the sequel, we need to consider an extension of the map $\varphi:S_1\times S_2\rightarrow \AA^1_{\lambda_0}\times S_2$ to a projective morphism. For that purpose,
we will construct a partial compactification of $S_1\times S_2$ by looking at a toric compactification
 of $S_1$. More precisely, consider the embedding
\begin{equation}\label{morphism:g}
\begin{array}{rcl}
g:S_1  & \hookrightarrow & \PP\left(V^\vee\right)=\PP^n \\
\underline{y}  & \longmapsto & (w_0:\ldots:w_n)=\left(1:\underline{y}^{\underline{a}_1}:\ldots: \underline{y}^{\underline{a}_n}\right)
\end{array}
\end{equation}
and put $X:=\overline{\text{Im}(g)}$. Consider the graph $\Gamma_f\subset S_1\times\AA^1_{\lambda_0}\times S_2$, and let $\Gamma X$ be the closure
of $\Gamma_f$ in $X\times\AA^1\times S_2$.
Then we have $\Gamma X\subset Z^*$, where
$$
\begin{tikzcd}
Z^*:= \left\{\sum_{i=0}^n \lambda_i\cdot w_i=0\right\}\ar[hook]{r}{i_P} & \PP^n\times \AA^1_{\lambda_0}\times S_2=:P
\end{tikzcd}
$$
is the universal hyperplane.
We thus have the commutative diagram
\begin{equation}\label{eq:CompactDiagram}
\begin{tikzcd}
S_1\times  S_2 \ar{rd}{\varphi} \ar[hook]{r} \ar[bend left=35]{rr}{k} & \Gamma X \ar{d}{\overline{\varphi}} \ar[hook]{r} & Z^* \ar[swap]{ld}{p_2} \ar[hook]{r}{i_P} & P, \ar{lld}{p_2}\\
& \AA^1\times S_2&
\end{tikzcd}
\end{equation}
where the map $S_1\times S_2 \hookrightarrow \Gamma X$ is the composition of the isomorphism $S_1\times S_2 \stackrel{\cong}{\ra}\Gamma_f$ with $\Gamma_f\hra\Gamma X$ and the map $\Gamma X \hookrightarrow Z^*$ is the closed embedding
mentioned above. Notice that $\overline{\varphi}$ is projective, and restricts to $\varphi$ on $S_1\times S_2$, i.e.,
it is the extension of $\varphi$ we were looking for.

We need also the following important geometric property of the morphisms $\varphi$ and $\overline{\varphi}$.
Recall (see \cite[\S8]{Sa2}) that for a smooth affine variety $U$, a regular function $h:U\rightarrow \AA^1$ with isolated
critical points is called cohomologically tame if there is a partial
compactification $j:U\hookrightarrow\overline{U}$ (i.e. $\overline{U}$ is quasi-projective) and an extension of $h$ to a projective morphism $\overline{h}:\overline{U}\rightarrow \AA^1$ such that the sheaf $\R j_* \QQ_U$ has no vanishing cycles with respect to $\overline{h}$ outside of $U$,
i.e., such that for any $c\in \AA^1$, we have
$$
\textup{supp}(\psi_{\overline{h}-c}(\R j_* \QQ_U)) \subset U.
$$
We call a morphism $H:\cX \rightarrow \mathcal{S}$ on a quasi-projective
variety $\cX$ stratified smooth if there is a locally finite stratification
by locally
closed smooth subvarieties such that the restriction of $H$ to any strata has no critical points.
\begin{lemma}\label{lem:StratSmooth}
We have the following properties of the partial compactifications defined above:
\begin{enumerate}
  \item For any parameter value $\ulambda\in S_2$, the morphism $f_{\ulambda}:S_1\rightarrow \AA^1_{\lambda_0}$ is cohomologically tame.
  \item  The morphism $\overline{\varphi}:\Gamma X \rightarrow \AA^1\times S_2$ is stratified smooth on its boundary $\Gamma X\setminus(S_1\times S_2) $.
  \item  Consider the composition $k':S_1\times S_2 \hookrightarrow P$ of the map $k:S_1\times S_2\hookrightarrow Z^*$ with
  the canonical closed embedding $i_P:Z^*\hookrightarrow P$. Then for any $\beta\in \CC^d$ and for $\star\in\{+,\dag\}$, the module $(k')_\star \cO^\beta_{S_1\times S_2}$ is non-characteristic with respect to $p_2$ along the boundary $\overline{\im(k')}\backslash\im(k')$, i.e., its characteristic variety  satisfies $\textup{char}(k_\star \cO^\beta_{S_1\times S_2}) \cap (\PP^n\times T^*(\AA^1_{\lambda_0}\times S_2))_{|T^*(P\backslash\im(k'))} \subset P\backslash\im(k')$.
\end{enumerate}
\end{lemma}
\begin{proof}
\begin{enumerate}
  \item According to Lemma \ref{lem:NonDegen}, $f_{\ulambda}$ is non-degenerate and convenient for any $\ulambda\in S_2$. It then follows from \cite[Lem. 3.4]{DeLoe} that $f_{\ulambda}$ is cohomologically tame, but since the notations in loc. cit. differ considerably
  from our situation, we recall the proof: We will show that the extension $\overline{f}_{\ulambda}:\Gamma X^{\ulambda}:=(\Gamma X)\cap(\PP^n\times\AA^1_{\lambda_0}\times\{\ulambda\})\rightarrow \AA^1_{\lambda_0}$ which is the restriction of $\overline{\varphi}$ over $\AA^1_{\lambda_0}\times\{\ulambda\}$ has no vanishing cycles at infinity.
  Recall that the projective toric variety $X$ is stratified by $X=\bigcup_{\delta\subset \Delta} X_\delta$, where $\Delta$ is the convex hull in $\RR^d$ of the columns of the matrix $A$, $\delta$ is a face of $\Delta$, and $X_\delta$ is a torus orbit associated with the face $\delta$ (cf., for instance, \cite[Prop. 5.1.9]{GKZbook}). We obtain an induced stratification $(\Gamma X_\delta)_{\delta\in\Delta}$ of $\Gamma X$ and an induced stratification $(\Gamma X^{\ulambda}_\delta)_{\delta\in\Delta}$ of the restriction $\Gamma X^{\ulambda}$.
    Then the smoothness of the restriction $(\overline{f}_{\ulambda})_{|\Gamma X^{\ulambda}_\delta}$  follows from the smoothness of the Laurent polynomial $f^\delta_{\ulambda}=\sum_{i:\underline{a}_i\in\delta} \lambda_i \underline{y}^{\underline{a}_i}$, which has no critical points
  since $f_{\ulambda}$ is non-degenerate (see Lemma \ref{lem:NonDegen}). Hence we see that $\overline{f}_{\ulambda}$ is
  stratified smooth on its boundary, i.e., outside $\Gamma X^{\ulambda}_\Delta = (\Gamma X \backslash \Gamma_f)\cap(\PP^n\times\AA^1_{\lambda_0}\times\{\ulambda\}) \cong S_1$. Then it follows from \cite[Prop. 4.2.8]{Di} that for any constructible complex $\cF^\bullet$ on $\Gamma X^{\ulambda}$ (with respect to the stratification $(\Gamma X^{\ulambda}_\delta)_{\delta\subset \Delta}$) we have $\psi_{\overline{f}_{\ulambda}-c} \cF^\bullet_{|(\Gamma X^{\ulambda} \backslash \Gamma X^{\ulambda}_\Delta)}=0$ for any $c\in \AA^1_{\lambda_0}$. Hence $\textup{supp}(\psi_{\overline{f}_{\ulambda}-c} \cF^\bullet) \subset \Gamma X^{\ulambda}_\Delta \cong S_1$, and so $f_{\ulambda}$ is cohomologically tame since we can apply this to the case where $\cF^\bullet = \R j_* \QQ_{S_1}$, $j:S_1\cong \Gamma X^{\ulambda}_\Delta \hookrightarrow
  \Gamma X^{\ulambda}$ being the canonical open embedding.
  \item The same proof as in (1) applies, i.e., the restriction $\overline{\varphi}_{|\Gamma X_\delta}:\Gamma X_\delta\rightarrow
  \AA^1_{\lambda_0}\times S_2$ is non-singular for all $\delta \subsetneq \Delta$.
  \item We have $\textup{char}((k')_\star\cO_{S_1\times S_2}^\beta) \subset
  \bigcup_{\delta\in\Delta} T^*_{\Gamma X_\delta}P$. The fact that $\overline{\varphi}_{|\Gamma X_\delta}$ is smooth (for $\delta\subsetneq \Delta$)
  means exactly that the fibres of $p_2:P\rightarrow \AA^1_{\lambda_0}\times S_2$ are transversal to $\Gamma X_\delta$. As
  the conormal bundle to these fibres is precisely the space $\PP^n\times T^*(\AA^1_{\lambda_0}\times S_2)$, we obtain
  $(\PP^n\times T^*(\AA^1_{\lambda_0}\times S_2))\cap \Gamma X_\delta \subset \PP^n\times \AA^1_{\lambda_0}\times S_2$, for
  all $\delta\subsetneq \Delta$, and so we have $\textup{char}(k_\star \cO^\beta_{S_1\times S_2}) \cap (\PP^n\times T^*(\AA^1_{\lambda_0}\times S_2))_{|T^*(P\backslash\im(k'))} \subset P\backslash\im(k')$,
  as required.
\end{enumerate}
\end{proof}

In order to achieve the comparison results of this section we need to use several variants of the Fourier-Laplace transformation; let us recall here the definitions.

\begin{defi}\label{def:FL}
Let $Y$ be a smooth algebraic variety, $U$ be a finite-dimensional complex vector space and $U'$ its dual vector space. Denote by $\cE$ the trivial vector bundle $\tau:U\times Y\ra Y$ and by $\cE'$ its dual.
Write $\can:U\times U'\ra\AA^1$ for the canonical morphism defined by $\can(a,\varphi)=\varphi(a)$. This extends to a function $\can:\cE\times_Y\cE'\ra\AA^1$. Define $\cL:=\cO_{\cE\times_{Y}\cE'}\cdot e^{-\can}$, the free rank one module with differential given by the product rule. Consider also the canonical projections $p_1:\cE\times_{Y}\cE'\ra\cE$, $p_2:\cE\times_{Y}\cE'\ra\cE'$. The partial Fourier-Laplace transformation is then defined by
$$\FL_{Y}(\bullet):=p_{2,+}\left(p_1^+(\bullet)\otimes_{\cO_{\cE\times_{Y}\cE'}}^\L\cL\right).$$
\end{defi}

If the base $Y$ is a point we recover the usual Fourier-Laplace transformation and we will simply write $\FL$. Notice that although this functor is defined at the level of derived categories, it is $t$-exact in the derived category of bounded complexes of $\cD$-modules with holonomic cohomologies, i.e., induces a functor $\FL_Y:\Mod_\text{h}(\cD_\cE)\ra \Mod_\text{h}(\cD_{\cE'})$.

We also need the following variant of the Fourier-Laplace transformation.
\begin{defi}\label{def:plocFL}
Keep the notations of the previous definition, and assume moreover that $U$ and $U'$ are one-dimensional, with respective coordinates $t$ and $\tau$. Put $z=\tau^{-1}$, and denote by $j_\tau:\mathds{G}_{m,\tau}\hra\AA^1_\tau$ and $j_z:\mathds{G}_{m,\tau}\hra\AA^1_z=\PP^1_\tau\setminus\{\tau = 0\}$ the canonical embeddings. Then the localized partial Fourier-Laplace transformation with respect to $\tau$ is defined by
$$\FL^{\loc}_{Y}:=(j_z\times\id_Y)_+(j_\tau\times\id_Y)^+\FL_{Y}.$$
\end{defi}

Our next aim is to compare the twisted de Rham cohomology of the family $\varphi:S_1\times S_2\rightarrow \AA^1_{\lambda_0}\times S_2$, i.e., the
left-hand side of the isomorphism of Proposition \ref{prop:TwdeRham} with an object derived from a certain intersection cohomology $\cD$-module on $Z^*$.
Recall that for a smooth algebraic variety $Y$, and an open subvariety $j:U\hra Y$, we call intersection cohomology module with coefficients in some $\cD_U$-module $\cN$ the intermediate extension $j_{\dag+}\cN:=\im(j_\dag \cN\rightarrow j_+\cN)$. Its name comes from the fact that, if $\cN$ is smooth and corresponds to a local system $\cL$ on $U$ under the Rieman-Hilbert correspondence, $j_{\dag+}\cN$ corresponds to the intersection cohomology complex on $Y$ with coefficients in $\cL$, $\IC(Y,\cL)$.

Now is when we can properly state and prove the next comparison result of this section. Recall that we take a matrix $A\in\text{M}(d\times n,\ZZ)$
satisfying the assumptions \ref{assump:AssMatrix} and a parameter vector $\beta\in \CC^d$. Recall that $Z^*=\left\{\sum_{i=0}^n w_i \lambda_i =0\right\}\subset P=\PP^n\times \AA^1_{\lambda_0}\times S_2$ denotes the universal hyperplane, the map $k$, introduced in diagram (\ref{eq:CompactDiagram}), is
$$
\begin{array}{rcl}
  k:S_1\times S_2 & \longrightarrow & Z^* \\
  (\underline{y},\underline{\lambda}) & \longmapsto & ((1:\underline{y}^{\underline{a}_1}:\ldots:\underline{y}^{\underline{a}_n}),
  (-\sum_{i=1}^{n}\lambda_i \underline{y}^{\underline{a}_1},\lambda_1,\ldots,\lambda_n))
\end{array}
$$
and $p_2$ is the restriction of the canonical projection from $P$ to $\AA^1\times S_2$ to the subspace $Z^*$.
\begin{prop}\label{prop:IsoAfterFourier}
In the above situation, we have the following isomorphism of $\cD_{\AA^1_z\times S_2}$-modules:
$$
\FL^{\loc}_{S_2}\cH^0p_{2,+}k_{\dag+}\cO_{S_1\times S_2}^\beta\stackrel{\sim}{\lra}\FL^{\loc}_{S_2}\cH^0\varphi_+\cO_{S_1\times S_2}^\beta.
$$
Moreover, this isomorphism is induced from the canonical morphism $k_{\dag+}\cO_{S_1\times S_2}^\beta\hra k_+\cO_{S_1\times S_2}^\beta$ by applying the functor $\FL^{\loc}_{S_2}\cH^0 p_{2,+}$.
\end{prop}

Before entering into the proof of this Proposition, we will need to state some facts about Radon transformations
for $\cD$-modules. We follow \cite[\S~2]{Reich2}. Recall that $V$ and $V^\vee$ are dual affine spaces of dimension $n+1$ with coordinates $w_0,\ldots,w_n$ and $\lambda_0,\ldots,\lambda_n$ respectively.
\begin{defi}
Denote by $Z\subset\PP(V^\vee)\times V$ the universal hyperplane with equation $\sum_{i=0}^nw_i\lambda_i=0$ and by $U:=(\PP(V^\vee)\times V)\setminus Z$ its complement. Consider the following commutative diagram:
$$\xymatrix{ && U \ar[drr]^{\pi_2^U} \ar[dll]_{\pi_1^U} \ar@{^(->}[d]^{j_U}&& \\ \PP(V^\vee) && \PP(V^\vee) \times V \ar[ll]_{\pi_1} \ar[rr]^{\pi_2} && V. \\ && Z \ar[ull]^{\pi_1^Z} \ar@{^(->}[u]_{i_Z} \ar[rru]_{\pi_2^Z} &&}$$
The Radon transformations are functors from $\Db(\cD_{\PP(V^\vee)})$ to $\Db(\cD_V)$ given by
\begin{align}
\cR & := \pi^Z_{2,+}\pi_1^{Z,+}\cong\pi_{2,+}i_{Z,+}i_{Z}^+\pi_1^+, \notag \\
\cR^\circ & := \pi^U_{2,+}\pi_1^{U,+}\cong\pi_{2,+}j_{U,+}j^+_U\pi_1^+, \notag \\
\cR^\circ_c & :=\pi^U_{2,\dag}\pi_1^{U,+}\cong\pi_{2,+}j_{U,\dag}j^+_U\pi_1^+, \notag \\
\cR_{cst} & :=\pi_{2,+}\pi_1^+. \notag
\end{align}
\end{defi}

\begin{prop}\label{prop:triangleRadon}
Let $g$ be as in (\ref{morphism:g}). Then, for every $\beta\in\CC^d$, we have the following two exact sequences of regular holonomic $\cD_V$-modules, which are dual to each other:
$$0\lra\cH^{-1}\cR_{cst}g_+ \cO_{S_1}^\beta\lra\cH^0\cR g_+\cO_{S_1}^\beta\lra\cH^0\cR^\circ_cg_+\cO_{S_1}^\beta\lra\cH^0\cR_{cst}g_+\cO_{S_1}^\beta\lra0,$$
$$0\lra\cH^0\cR_{cst}g_\dag\cO_{S_1}^{-\beta}\lra\cH^0\cR^\circ g_\dag\cO_{S_1}^{-\beta}\lra\cH^0\cR g_\dag\cO_{S_1}^{-\beta}\lra\cH^1\cR_{cst}g_\dag\cO_{S_1}^{-\beta}\lra0.$$
Moreover, the $\cD_V$-modules $\cH^i\cR_{cst}g_\star \cO_{S_1}^\beta$, for $i\in \{-1,0,1\}$ and $\star\in\{+,\dag\}$, that appear in the above sequences are $\cO_V$-free. Consequently, for any $\beta\in\CC^d$, calling $j_{V^*}$ the canonical inclusion $V^*:=\AA^1_{\lambda_0}\times S_2\hra V$, we have isomorphisms of $\cD_{\AA^1_z\times S_2}$-modules
$$\begin{array}{ccc}
\FL_{S_2}^{\loc}j_{V^*}^+\cR^\circ_c g_+\cO_{S_1}^\beta & \cong & \FL_{S_2}^{\loc}j_{V^*}^+\cR g_+\cO_{S_1}^\beta,\\[3pt]
\FL_{S_2}^{\loc}j_{V^*}^+\cR^\circ g_\dag\cO_{S_1}^\beta & \cong & \FL_{S_2}^{\loc}j_{V^*}^+\cR g_\dag\cO_{S_1}^\beta
\end{array}.$$
\end{prop}
\begin{proof}
Following the notation of the previous definition, since $Z$ is smooth, the excision triangle (\hspace{-.5pt}\cite[Prop. 1.7.1]{Hotta}) corresponding to the diagram $U\hra\PP(V^\vee)\times V\leftarrow Z$ gives rise to the following triangles of Radon transformations for any $\cM\in\Dbc(\cD_{\PP(V^\vee)})$:
\begin{align}
\cR_{cst}\cM\lra\cR^\circ\cM\lra\cR\cM\lra, \label{eq:Radontri1}\\
\cR\cM\lra\cR^\circ_{c}\cM\lra\cR_{cst}\cM\lra \label{eq:Radontri2},
\end{align}
where the second triangle is dual to the first.

Note that $\DD\cO_{S_1}^\beta\cong\cO_{S_1}^{-\uuno-\beta}\cong\cO_{S_1}^{-\beta}$. It suffices then to show the existence of the first exact sequence of the statement of the proposition and the fact that the $\cH^i\cR_{cst}g_\star \cO_{S_1}^\beta$ are constant. The existence of the exact sequences here follows by just a variation of \cite[Prop. 2.8]{Reich2}, but considering the twist by $\uy^{-1-\beta}$ of $\cO_{S_1}$. That needs in turn [ibid., Prop. 2.5, 2.7, Lem. 2.6], but almost every argument is functorial and anyway they can be easily adapted to our context. The constancy of $\cH^i\cR_{cst}g_\star\cO_{S_1}^{\beta}$ can be proved as the second point of [ibid., Lem. 2.7].
\end{proof}

\begin{proof}[Proof of Proposition \ref{prop:IsoAfterFourier}]
Call $p_1$ the restriction to $Z^*$ of the projection $\PP^n\times\AA_{\lambda_0}^1\times S_2\ra\PP^n$, abusing a bit of the notation.
Consider the cartesian diagram
$$
\xymatrix{\ar @{} [dr] |{\Box} S_1\times S_2 \ar[d]_{k} \ar[r]^{\pi_1} & S_1 \ar[d]^{g}\\
Z^* \ar[r]_{p_1} & \PP^n}.
$$
By base change, we have that
\begin{equation}\label{eq:BaseChange1}
k_+\cO_{S_1\times S_2}^\beta\cong p_1^+g_+\cO_{S_1}^\beta
\end{equation}
for any $\beta\in\CC^d$. Now, since $p_1$ is smooth, we have the analogous isomorphism $k_\dag\cO_{S_1\times S_2}^\beta\cong p_1^+g_\dag\cO_{S_1}^\beta$ for every $\beta$ as well. Now note that because of $p_1$ being non-characteristic, it is easy to show (cf. the proof of the third point of \cite[Prop. 2.22]{ReiSe2}) that for every $\beta$
\begin{equation}\label{eq:BaseChange2}
k_{\dag+}\cO_{S_1\times S_2}^\beta\cong p_1^+g_{\dag+}\cO_{S_1}^\beta.
\end{equation}

Consider now the commutative diagram
$$\xymatrix{ & Z^* \ar[dl]_{p_1} \ar[r]^{p_2} \ar[d]^{j_{Z^*}} & V^* \ar[d]^{j_{V^*}}\\
\PP^n & Z \ar[l]^{\pi_1^Z} \ar[r]_{\pi_2^Z} & V},$$
where the square is also cartesian. By applying base change again, we obtain the natural transformations of functors from $\Dbc(\cD_{\PP^n})$ to $\Dbc(\cD_{V^*})$
\begin{equation}\label{eq:NatTr}
p_{2,+}p_1^+\cong p_{2,+}j_{Z^*}^+\pi_1^{Z,+}\cong j_{V^*}^+\pi_{2,+}^Z\pi_1^{Z,+}=j_{V^*}^+\cR.
\end{equation}
Therefore, applying $p_{2,+}$ to isomorphisms \eqref{eq:BaseChange1} and \eqref{eq:BaseChange2} (and taking into account
that $\varphi=p_2\circ k$, see diagram \eqref{eq:CompactDiagram}) we obtain
\begin{equation}\label{eq:MIC_Project}
\varphi_+\cO^\beta_{S_1 \times S_2} \cong p_{2,+}k_{+}\cO_{S_1\times S_2}^\beta\cong j_{V^*}^+\cR g_{+}\cO_{S_1}^\beta,
\quad\textup{and}\quad
p_{2,+}k_{\dag+}\cO_{S_1\times S_2}^\beta\cong j_{V^*}^+\cR g_{\dag+}\cO_{S_1}^\beta,
\end{equation}
for any $\beta\in\CC^d$.

We need now to relate the various Radon transformations with the Fourier-Laplace transformation $\FL: \Mod_{\text{h}}(\cD_{V^\vee})\ra\Mod_{\text{h}}(\cD_V)$. This is possible due to the fundamental result of d'Agnolo and Eastwood \cite[Prop. 1]{AE}. We quote the formulation from the proof of \cite[Lem. 2.12]{ReiSe2}. Let $Bl_0\left(V^\vee\right)\subset\PP^n\times V^\vee$ be the blow-up of $V^\vee$ at the origin and consider the commutative diagram
$$\xymatrix{ T:=\GG_m\times S_1 \ar [r]^{h} \ar [dr]^{\tilde{h}} \ar[dd]_{\pi_T} & V^\vee & Bl_0(V^\vee) \ar[l]_p \ar[ddl]^{q}\\
& V^\vee\setminus\{0\} \ar[u]_{j_0} \ar[d]_{\pi} & \\
S_1 \ar[r]^{g} & \PP^n &},$$
where $\pi$ is the canonical morphism $V^\vee\setminus\{0\}\ra\PP(V^\vee)$, $\pi_T$ is the second projection, and $h$ and $\tilde{h}$ are given by $(y_0,\uy)\mapsto (y_0,y_0\uy^{\underline{a}_1},\ldots,y_0\uy^{\underline{a}_d})$. Then we have the natural transformations
\begin{equation}\label{eq:NatTr2}
\cR^\circ_cg_+\cong\FL h_+\pi_T^+\,\text{ and }\,\cR^\circ g_\dag\cong\FL h_\dag\pi_T^+,
\end{equation}
of functors from $\Mod_{\text{h}}(\cD_{S_1})$ to $\Mod_{\text{h}}(\cD_V)$. Notice that although
in general $\cR^\circ_c$ and $\cR^\circ$ are not exact, so the compositions $\cR^\circ_c g_+$ and $\cR^\circ g_\dag$ are precisely due to the above isomorphisms.

Applying $\FL_{S_2}^{\loc}j_{V^*}^+$ to the natural transformations in (\ref{eq:NatTr2}) we obtain the following ones of functors from $\Mod_{\text{h}}(\cD_{S_1})$ to $\Mod_{\text{h}}(\cD_{\AA^1_z\times S_2})$:
\begin{equation}\label{eq:RadonFourier}
\FL_{S_2}^{\loc}j_{V^*}^+\cR^\circ_c g_+\cong\FL_{S_2}^{\loc}j_{V^*}^+\FL h_+\pi^+\,\text{ and }\,\FL_{S_2}^{\loc}j_{V^*}^+\cR^\circ g_\dag\cong\FL_{S_2}^{\loc}j_{V^*}^+\FL h_\dag\pi^+
\end{equation}
We are now closer to the isomorphism of the statement; let us rewrite in a different way the functors above to obtain it. Namely,
$$\FL^{\loc}_{S_2}j_{V^*}^+\FL\cong j_{V^*}^+\FL^{\loc}_W\FL\cong j_{V^*}^+(j_z\times\id_W)_+(j_\tau\times\id_W)^+\FL_W\FL\cong
j_{V^*}^+(j_z\times\id_W)_+(j_\tau\times\id_W)^+\FL_{\AA^1_{w_0}}.$$
From the third of the assumptions \ref{assump:AssMatrix} we know that we can decompose the morphism $h$ as
$$\begin{tikzcd}
h:T \ar[hook]{r}{h_1} & \mathds{G}_{m,w_0}\times W^\vee \ar[hook]{rr}{j_\tau\times\id_{W^\vee}} &&
\AA^1_{w_0}\times W^\vee=V^\vee
\end{tikzcd},$$
where $h_1$ is a closed embedding (cf. \cite[Prop. 2.1 (1), proof of Thm. 2.4, p. 213]{ReiSe}). (Recall that $w_0=-\tau$.) It follows that
$$(j_\tau\times\id_W)^+\FL_{\AA^1_{w_0}}h_+\cong\FL_{\AA^1_{w_0}}(j_\tau\times\id_{W^\vee})^+h_+\cong
\FL_{\AA^1_{w_0}}(j_\tau\times\id_{W^\vee})^+(j_\tau\times\id_{W^\vee})_+h_{1,+}\cong\FL_{\AA^1_{w_0}}h_{1,+}.$$
Summarizing, we obtain that
$$\FL^{\loc}_{S_2}j_{V^*}^+\FL h_\star\cong j_{V^*}^+(j_z\times\id_W)_+\FL_{\AA^1_{w_0}}h_{1,+},$$
where $\star\in\{+,\dag\}$, because $h_{1,+}=h_{1,\dag}$ for $h_1$ is proper. In particular,
we have $\FL^{\loc}_{S_2}\circ j_{V^*}^+ \circ \FL \circ h_+ \cong \FL^{\loc}_{S_2}\circ j_{V^*}^+ \circ \FL \circ h_\dag$. As a consequence, we can claim using the isomorphisms in (\ref{eq:RadonFourier}) that
$$\FL_{S_2}^{\loc}j_{V^*}^+\cR^\circ_cg_+\cong\FL_{S_2}^{\loc}j_{V^*}^+\cR^\circ g_\dag,$$
and so, from Proposition \ref{prop:triangleRadon} and applying this last natural transformation to $\cO_{S_1}^\beta$,
$$\FL_{S_2}^{\loc}j_{V^*}^+\cR g_{\dag+}\cO_{S_1}^\beta\cong\FL_{S_2}^{\loc}j_{V^*}^+\cR g_+\cO_{S_1}^\beta.$$

In conclusion, we finally obtain, by using the isomorphisms from (\ref{eq:MIC_Project}) as well as the exactness
of the functor $\FL^{\loc}_{S_2}$, that
$$
\FL_{S_2}^{\loc}\cH^0p_{2,+}k_{\dag+}\cO_{S_1\times S_2}^\beta\cong \FL_{S_2}^{\loc}\cH^0\varphi_+\cO_{S_1\times S_2}^\beta.
$$

The last statement is an easy consequence of isomorphisms in (\ref{eq:NatTr}) and (\ref{eq:MIC_Project}), noting that
$$\FL^{\loc}_{S_2}\cH^0p_{2,+}k_{\dag+}\cO_{S_1\times S_2}^\beta\cong\FL^{\loc}_{S_2}\cH^0j_{V^*}^+\cR g_{\dag+}\cO_{S_1}^\beta\cong\FL_{S_2}^{\loc}j_{V^*}^+\cR g_+\cO_{S_1}^\beta\cong\FL_{S_2}^{\loc}\cH^0\varphi_+\cO_{S_1\times S_2}^\beta$$
is induced from the canonical morphism $g_{\dag+}\cO_{S_1}^\beta\hra g_+\cO_{S_1}^\beta$ via $\FL_{S_2}^{\loc}j_{V^*}^+\cR$.
\end{proof}

Consider again the space $P=\PP^n\times(\AA^1_{\lambda_0}\times S_2)$, together with the canonical closed embedding
$i_P:Z^* \hookrightarrow P$. We have the diagram
$$
\begin{tikzcd}
Z^* \ar{d}{p_2} \ar[hook]{r}{i_P} & P \ar{ld}{p_2} \\
\AA^1_{\lambda_0}\times S_2
\end{tikzcd},
$$
where we denote the restriction of the projection $p_2:P\rightarrow \AA^1_{\lambda_0}\times S_2$ to $Z^*$ by
the same letter (as was done before).
Since $i_P$ is proper, we have
$$
\begin{array}{c}
\cH^0p_{2,+}k_{\dag+}\cO_{S_1\times S_2}^\beta \cong \cH^0p_{2,+}(i_P\circ k)_{\dag+}\cO_{S_1\times S_2}^\beta
\quad
\text{ and }
\\\\
\cH^0\varphi_+\cO_{S_1\times S_2}^\beta
\cong \cH^0p_{2,+}k_+\cO_{S_1\times S_2}^\beta \cong
\cH^0p_{2,+}(i_P\circ k)_+\cO_{S_1\times S_2}^\beta.
\end{array}
$$
so that we can also write the morphism
$\cH^0p_{2,+}k_{\dag+}\cO_{S_1\times S_2}^\beta\lra\cH^0\varphi_+\cO_{S_1\times S_2}^\beta$
from the proof of Proposition \ref{prop:IsoAfterFourier} as
\begin{equation}\label{eq:MorphismPhiBeforeReduction}
\cH^0p_{2,+}(i_P\circ k)_{\dag+}\cO_{S_1\times S_2}^\beta\lra
\cH^0p_{2,+}(i_P\circ k)_+\cO_{S_1\times S_2}^\beta.
\end{equation}

In order to proceed, we will have to take into account certain group actions on the spaces $S_1\times S_2$,
$P$ and $\AA^1_{\lambda_0}\times S_2$ as well as some equivariance properties
of the various sheaves of modules on these spaces. For the reader's convenience, we recall some facts from
\cite[\S~2.4]{ReiSe2}.

We consider the action of $S_1$ on $S_2$ given by
$$
\begin{array}{rcl}
\mu:S_1\times S_2 & \longrightarrow & S_2 \\ \\
(y_1,\ldots,y_d,\lambda_1,\ldots,\lambda_n) & \longmapsto & (\underline{t}^{-\underline{a}_1}\lambda_1,\ldots,\underline{t}^{-\underline{a}_n}\lambda_n),
\end{array}
$$
inducing the following three:
$$
\begin{array}{rcl}
S_1\times \left(S_1\times S_2\right) & \longrightarrow & S_1 \times S_2 \\ \\
(\underline{t},(y_1,\ldots,y_d,\lambda_1,\ldots,\lambda_n)) & \longmapsto & (t_1 y_1,\ldots,t_d y_d,\underline{t}^{-\underline{a}_1}\lambda_1,\ldots,\underline{t}^{-\underline{a}_n}\lambda_n),
\end{array}
$$
$$
\begin{array}{rcl}
S_1\times P & \longrightarrow & P= \PP^n\times(\AA^1_{\lambda_0}\times S_2) \\ \\
(\underline{t},((w_0:\ldots:w_n),\lambda_0,\lambda_1,\ldots,\lambda_n)) & \longmapsto & ((w_0:\underline{t}^{\underline{a}_1} w_1:
\ldots:\underline{t}^{\underline{a}_n} w_n),\lambda_0, \underline{t}^{-\underline{a}_1}\lambda_1,\ldots,\underline{t}^{-\underline{a}_n}\lambda_n),
\end{array}
$$
$$
\begin{array}{rcl}
S_1\times \left(\AA^1_{\lambda_0} \times S_2\right) & \longrightarrow & \AA^1_{\lambda_0} \times S_2 \\ \\
(\underline{t},(\lambda_0,\lambda_1,\ldots,\lambda_n)) & \longmapsto & (\lambda_0, \underline{t}^{-\underline{a}_1}\lambda_1,\ldots,\underline{t}^{-\underline{a}_n}\lambda_n).
\end{array}
$$

It is easy to see that all these four actions are free, basically because $\mu$ is free and that they all have smooth geometric
quotients.  These are described
by the following result, which we cite from \cite[§2.4]{ReiSe2}.  For a free action of an algebraic group $G$ on a smooth variety $X$ admitting a geometric quotient, we write $X/G$ for such quotient.
\begin{prop}
In the above situation, put $\Sred:=(\gmt)^{n-d}$. Then the geometric quotients
$(S_1\times S_2)/S_1$, $P/S_1$
and $(\AA^1_{\lambda_0}\times S_2)/S_1$ are given, respectively, by the spaces
$$
S_1\times \Sred, \; \Pred:=\PP^n\times \AA^1_{\lambda_0} \times \Sred, \;\text{ and }
\AA^1_{\lambda_0}\times \Sred.
$$
There is a canonical embedding $\Sred \hookrightarrow S_2$ inducing embeddings (all denoted by $\iota$)
$$
\begin{array}{rcl}
S_1\times \Sred & \hookrightarrow  &S_1\times S_2,\\
\Pred & \hookrightarrow  &P=\PP^n\times \AA^1_{\lambda_0} \times S_2,\\
S_1\times \AA^1_{\lambda_0} \times \Sred & \hookrightarrow  &S_1\times \AA^1_{\lambda_0} \times S_2.
\end{array}
$$
\end{prop}
In the sequel, we will always consider $\Sred$ as a subvariety of $S_1$ (as well as
$\Pred$ as a subvariety of $P$ etc.).

There is a more direct description of the construction of $\Sred$ resp. of the embedding $\iota$ which we recall for the reader's convenience. Namely, let $B$ a Gale dual of $A$ (that is, an integer matrix $B\in\text{M}(n\times (n-d),\ZZ)$ whose columns generate $\ker_\QQ(A)$; see Proposition \ref{prop:hypGKZ}). Then we have the split exact sequence of abelian groups
$$
0\lra\ZZ^{n-d}\stackrel{B\cdot}{\lra}\ZZ^n\stackrel{-A\cdot}{\lra}\ZZ^d\lra0.
$$
By applying the functor $\text{Hom}_{\ZZ}(\bullet,\gm)$ to that sequence, we get the exact sequence of algebraic groups
$$
(1)\lra S_1=\gm^d \lra S_2=\gm^n \lra\gm^{n-d}\lra(1),
$$
where the first nontrivial morphism is just $\mu(\bullet,\uuno)$. From this it follows that the torus $\gm^{n-d}$ can be canonically
identified with the geometric quotient $S_2/S_1$. Since this exact sequence of algebraic tori also splits, we can chose a section $\iota:\Sred = S_2/S_1\hra S_2$, and this choice corresponds exactly to the one in \cite[§2.4]{ReiSe2}.

We define
reduced versions of the maps $\varphi$ and $i_P\circ k$ by the cartesian diagrams
$$
\begin{tikzcd}
S_1 \times S_2 \ar{rr}{\varphi} & & \AA^1_{\lambda_0}\times S_2 \\ \\
S_1 \times \Sred \ar[hook]{uu}{\iota} \ar{rr}{{\,'\!}\varphi} && \AA^1_{\lambda_0}\times \Sred \ar[hook]{uu}{\iota}
\end{tikzcd}
\quad\quad
\text{and}
\quad\quad
\begin{tikzcd}
S_1 \times S_2 \ar{rr}{i_P\circ k} & & P \\ \\
S_1 \times \Sred \ar[hook]{uu}{\iota} \ar{rr}{l} && \Pred\ar[hook]{uu}{\iota}
\end{tikzcd}.
$$

The mapping ${\,'\!}\varphi: S_1\times \Sred\rightarrow \AA^1_{\lambda_0}\times \Sred$ is a family of Laurent polynomials
parametrized by $\Sred$, and we denote by ${\,'\!}f:S_1\times \Sred\rightarrow \AA^1_{\lambda_0}$ its first component.
In order to illustrate these statements, consider the main case of interest, where we have
$$
A=
\begin{pmatrix}
1 & -1 & 0 & 0 &\ldots  & 0  \\
1 &  0 & -1 & 0 & \ldots & 0 \\
\vdots &\ldots \\
1 &  0 & 0 & 0 & \ldots & -1
\end{pmatrix}
\in \text{M}\left((n-1) \times n, \ZZ\right).
$$
Then $d=n-1$, $\Sred=\GG_{m,t}$ and we can choose
$$
\begin{array}{rcl}
\iota: \Sred=\gmt & \longrightarrow &  S_2 = \gm^n\\ \\
t & \longmapsto & (t,1,\ldots,1).
\end{array}
$$
Then we have (see also the explanation on page \pageref{page:Roadmap})
\begin{equation}\label{eq:LaurPolPn}
\begin{array}{rcl}
{\,'\!}\varphi:S_1\times \GG_{m,t} & \longrightarrow & \AA^1_{\lambda_0}\times \GG_{m,t} \\ \\
(y_1,\ldots,y_{n-1},t) & \longmapsto & \left(-\frac{1}{y_1}-\ldots-\frac{1}{y_{n-1}}-t\cdot y_1\cdot\ldots\cdot y_{n-1}, t \right)
\end{array}
\end{equation}

Going back to the situation of a general matrix $A\in \text{M}(d\times n,\ZZ)$ satisfying the assumptions
\ref{assump:AssMatrix}, we state the following result, inspired from \cite{ReiSe2}, showing that $\iota$ behaves well with respect
to all modules in question.
To simplify our notation we will write
$$
M_{\dag+,S_2}:=\cH^0 p_{2,+} (i_P\circ k)_{\dag\,+}\cO^\beta_{S_1\times S_2}
\quad\quad
\text{and}
\quad\quad
M_{S_2}:=\cH^0\varphi_+\cO^\beta_{S_1\times S_2} =
\cH^0 p_{2,+} (i_P\circ k)_+\cO^\beta_{S_1\times S_2}
,$$
and analogously,
$$
M_{\dag+}:=\cH^0 {'\!}p_{2,+} l_{\dag\,+}\cO^\beta_{S_1\times \Sred}
\quad\quad
\text{and}
\quad\quad
M:=\cH^0{\,'\!}\varphi_+\cO^\beta_{S_1\times \Sred} =
\cH^0 p_{2,+} l_+\cO^\beta_{S_1\times \Sred},
$$
where we write ${'\!}p_2:\Pred=\PP^n\times\AA^1_{\lambda_0}\times \Sred\rightarrow \AA^1_{\lambda_0}\times \Sred$
(to distinguish this projection from the one to $\AA^1_{\lambda_0}\times S_2$ considered above).
\begin{prop}\label{prop:IsoAfterFourierRed}
We have a morphism of $\cD_{\AA^1_{\lambda_0}\times \Sred}$-modules
$$
\phi:M_{\dag +} \longrightarrow M
$$
inducing an isomorphism of $\cD_{\AA^1_z\times \Sred}$-modules
$$
\FL^{\loc}_{\Sred} M_{\dag +}\cong
\FL^{\loc}_{\Sred} M.
$$
Moreover, both modules $M$ and $M_{\dag+}$ are non-characteristic with respect to the projection ${'\!}p:{'\!}P\rightarrow \AA^1_{\lambda_0}\times {'\!}S$ along the subspace ${'\!}P\, \backslash\, \Pred^*$, where $\Pred^*=\AA^n\times(\AA_{\lambda_0}^1\times \Sred)$.
\end{prop}
\begin{proof}
It has been shown in {\cite[Prop. 2.22., Lem. 6.4.]{ReiSe2}} that the morphism $\iota:\AA^1_{\lambda_0}\times \Sred\hookrightarrow \AA^1_{\lambda_0}\times S_2$ is non-characteristic for
the modules $M_{\dag +,S_2}$ and $M_{S_2}$. Consequently,
we obtain the morphism of $\cD_{\AA^1_{\lambda_0}\times \Sred}$-modules
$$
\phi:M_{\dag +} \longrightarrow M
$$
by applying $\iota^+$ to the morphism
$$
M_{\dag +,S_2} = \cH^0 p_{2,+} (i_P\circ k)_{\dag\,+}\cO^\beta_{S_1\times S_2}
\longrightarrow \cH^0\varphi_+\cO^\beta_{S_1\times S_2} =
\cH^0 p_{2,+} (i_P\circ k)_+\cO^\beta_{S_1\times S_2} = M_{S_2}
$$
(see \eqref{eq:MorphismPhiBeforeReduction}).
Since both the functor $\FL^{\loc}_{S_2}$ and the functor
$\FL^{\loc}_{\Sred}$ commute with $\iota$, we obtain the isomorphism
$\FL^{\loc}_{\Sred} M_{\dag +}\cong \FL^{\loc}_{\Sred} M$.

Finally, the fact that $\iota$ is non-characteristic for $M_{\dag +,S_2}$ and $M_{S_2}$
yields the last statement using Lemma
\ref{lem:StratSmooth}.
\end{proof}

\begin{rem}\label{rem:realExponents}
From now on we will assume that the $\beta_i$ are real numbers, since we will use some Hodge theoretic constructions,
which are not valid for arbitrary complex $\beta_i$, as commented in Remark \ref{rem:KummerHodge}.
\end{rem}
We start by introducing certain auxiliary filtrations on the $\cD$-modules considered above which are not a priori
Hodge filtrations but have an explicit description. We will see later that it is sufficient to consider these
filtrations to extract the Hodge theoretic information we need.

Recall that an easy calculation decomposing ${\,'\!}\varphi$ as a graph embedding followed by a projection shows that the direct image complex ${\,'\!}\varphi_+\cO_{S_1\times \Sred}^\beta$ can be represented by
$$
\Big({\,'\!}\varphi_*\Omega^{\bullet+d}_{S_1\times \Sred/\Sred}[\dd_{\lambda_0}], d-\kappa(\beta)\wedge-(d{\,'\!}f\wedge)\otimes \dd_{\lambda_0}\Big),
$$
where ${\,'\!}f$ is still the first component of ${\,'\!}\varphi$.

We consider the filtration on each ${\,'\!}\varphi_*\Omega^{l+d}_{S_1\times \Sred/\Sred}[\dd_{\lambda_0}]$ given by
\begin{equation}\label{eq:FiltDirectImageComp}
F_{k+l}{\,'\!}\varphi_*\Omega^{l+d}_{S_1\times \Sred/\Sred}[\dd_{\lambda_0}]=\sum_{i=0}^{k+l}{\,'\!}\varphi_*\Omega^{l+d}_{S_1\times \Sred/\Sred}\partial_{\lambda_0}^i.
\end{equation}
One easily checks that this filtration is compatible with the differential, so that we obtain the filtered complex
$$
\begin{array}{c}
F_k\Big({\,'\!}\varphi_*\Omega^{\bullet+d}_{S_1\times \Sred/\Sred}[\dd_{\lambda_0}], d-(\kappa(\beta)\wedge)-(d{\,'\!}f\wedge)\otimes \dd_{\lambda_0}\Big):=\\
\Big(F_{k+\bullet}{\,'\!}\varphi_*\Omega^{\bullet+d}_{S_1\times \Sred/\Sred}[\dd_{\lambda_0}], d-(\kappa(\beta)\wedge)-(d{\,'\!}f\wedge)\otimes \dd_{\lambda_0}\Big).
\end{array}
$$
\begin{defi}
We call the induced filtration $F_\bullet \cH^i{\,'\!}\varphi_+\cO_{S_1\times \Sred}^\beta$ on the cohomology sheaves $\cH^i{\,'\!}\varphi_+\cO_{S_1\times \Sred}^\beta$ the \emph{Brieskorn filtration}.
\end{defi}
Notice that for $i=0$, $k=0$ and $\beta=0$ the filtration step $F_k \cH^0{\,'\!}\varphi_+\cO_{S_1\times \Sred}^\beta =F_0 M$
is exactly the module $M_0$ considered in \cite{Sa2} and was called Brieskorn lattice there
because it was defined similarly to the case of isolated hypersurface singularities (see \cite{Brie}).

On the other hand, it is well-known that $l_{\dag+}\cO_{S_1\times \Sred}^\beta$ carries a filtration $F^Hl_{\dag+}\cO_{S_1\times \Sred}^\beta$ such that $\left(l_{\dag+}\cO_{S_1\times \Sred}^\beta, F^H\right)$ underlies a pure polarizable complex Hodge module of weight $n=\dim(S_1)+\dim(\Sred)$, in particular, an element of $\MHM(\Pred,\CC)$.
Since ${'\!}p_2$ is projective, the direct image $M_{\dag +}=\cH^0 {'\!}p_{2,+}l_{\dag+}\cO_{S_1\times \Sred}^\beta$ also carries a Hodge filtration $F^H\cH^0 {'\!}p_{2,+}l_{\dag+}\cO_{S_1\times \Sred}^\beta$ such that $\left(M_{\dag +},F^H\right)$ underlies a pure polarisable complex Hodge module (see the first point of \cite[Thm. 1]{Saito1}).
We consider the shifted filtration $F^{H_{sh}}_\bullet M_{\dag +}:=F^H_{\bullet-d} M_{\dag +}$. The next step is to compare this filtration to the Brieskorn filtration on $M$ via the morphism $\phi$. We will show in the next lemma that $\phi$ is filtered.
Notice that since the Brieskorn filtration has a priori no Hodge properties, we cannot simply deduce this result
from the functorial properties of mixed Hodge modules.

\begin{lemma}\label{lem:InclusionHodgeFilt}
Consider again the morphism
$\phi:M_{\dag +} \longrightarrow M$ from Proposition \ref{prop:IsoAfterFourierRed}
which yields
an isomorphism of $\cD_{\AA^1_z\times \Sred}$-modules after applying the functor $\FL^{\loc}_{\Sred}$.
Then $\phi$ is filtered with respect to the shifted Hodge filtration on $M_{\dag+}$ and with respect
to the Brieskorn filtration on $M$, i.e., for every $k\in\ZZ$, we have the inclusion
$$
\phi\left(F^{H_{sh}}_k M_{\dag +} \right)\subset F_k M.
$$
\end{lemma}
\begin{proof}
Recall that $\Pred^*=\AA^n\times(\AA_{\lambda_0}^1\times \Sred)$ is the complement in $\Pred$ of the divisor
$w_0=0$. Then
the map $l$ can be decomposed into a closed embedding
$l_1:S_1\times \Sred \hookrightarrow \Pred^*$ and an open embedding $l_2: \Pred^*\hookrightarrow \Pred$.
Here again the fact that $l_1$ is closed follows from the third assumption in \ref{assump:AssMatrix}.
We can further decompose $l_1$ as $l_1=l_0\circ i_{{\,'\!}\varphi}$, where $i_{{\,'\!}\varphi}$ is the graph embedding
of ${\,'\!}\varphi$.
More precisely,
we have the following diagram:
\begin{equation}\label{eq:CompactDiagram-2}
\begin{tikzcd}
S_1\times  \Sred \ar[swap]{rrdd}{{\,'\!}\varphi} \ar[hook]{r}{i_{{\,'\!}\varphi}} \ar[swap,bend left=35]{rr}{l_1}
\ar[bend left=35]{rrr}{l}& S_1\times(\AA^1_{\lambda_0}\times \Sred) \ar[hook]{r}{l_0} \ar{rdd}{{'\!}p_2}&
\Pred^*=\AA^n\times(\AA^1_{\lambda_0}\times \Sred) \ar{dd}{{'\!}p_2} \ar[hook]{r}{l_2}& \Pred=\PP^n\times(\AA^1_{\lambda_0}\times \Sred) \ar{ldd }{{'\!}p_2} \\ \\
& & \AA^1_{\lambda_0}\times \Sred&
\end{tikzcd},
\end{equation}
where we denote by slight abuse of notation all projections to the last coordinates by ${'\!}p_2$.
Recall that $M_{\dag+}=\cH^0 {'\!}p_{2,+} l_{\dag\,+}\cO^\beta_{S_1\times \Sred}$ and $M=\cH^0{\,'\!}\varphi_+\cO^\beta_{S_1\times \Sred}
=\cH^0{'\!}p_{2,+} l_{1,+}\cO^\beta_{S_1\times \Sred}$.

Notice that the Hodge filtration $F^H_\bullet M_{\dag +}$
is induced from the filtration $F^H_\bullet \text{DR}^{\bullet+n}_{\Pred/\AA_{\lambda_0}^1\times \Sred}(l_{\dag +}\cO_{S_1\times \Sred}^\beta)$
on the relative de Rham complex, where
\begin{equation}\label{eq:HodgeFiltdeRham1}
F^H_k \text{DR}^{\bullet+n}_{\Pred/\AA_{\lambda_0}^1\times \Sred}(l_{\dag +}\cO_{S_1\times \Sred}^\beta)
:=\left(\ldots\lra \Omega^{n+l}_{\Pred/\AA_{\lambda_0}^1\times \Sred}\otimes F^H_{k+n+l} l_{\dag +}\cO_{S_1\times \Sred}^\beta \lra \ldots \right)
\end{equation}
by the first point of \cite[Thm. 1]{Saito1}.

On the other hand, we have the isomorphism of $\cD_{\AA^1_{\lambda_0}\times \Sred}$-modules
$$
\cH^0{'\!}p_{2,*} \text{DR}^{n+\bullet}_{\Pred^*/\AA_{\lambda_0}^1\times \Sred}\left(l_{1, +}\cO_{S_1\times \Sred}^\beta\right) \cong M
$$
as the above diagram shows, just by the definition of $\cH^0{'\!}p_{2,+}$.
The Hodge filtration $F_\bullet^H$ on the module $l_{1, +}\cO_{S_1\times \Sred}^\beta$ can be explicitly written down, due to the fact that $l_1$ is a closed embedding (in particular, the filtered module
$(l_{1, +}\cO_{S_1\times \Sred}^\beta, F_\bullet^H)$ underlies a pure Hodge module), and then one checks that the shifted Brieskorn filtration $F_\bullet[-d]:=F_{\bullet+d}$ on the module $M$ (see formula \eqref{eq:FiltDirectImageComp} for the definition of the Brieskorn filtration) is induced
from the filtration
\begin{equation}\label{eq:HodgeFiltdeRham2}
F_k^H \text{DR}^{\bullet+n}_{\Pred^*/\AA_{\lambda_0}^1\times \Sred}(l_{1, +}\cO_{S_1\times \Sred}^\beta)
:=\left(\ldots\lra \Omega^{n+l}_{\Pred^*/\AA_{\lambda_0}^1\times \Sred}\otimes F^H_{k+n+l} l_{1, +}\cO_{S_1\times \Sred}^\beta \lra \ldots \right).
\end{equation}
Notice once again that we cannot deduce from this description that the Brieskorn filtration $F_\bullet M$ (or its shifted version $F_\bullet[-d] M$)
has any Hodge properties since the map ${'\!}p_2:\Pred^* \rightarrow  \AA^1_{\lambda_0}\times \Sred$  is not projective
(and in particular the filtration from formula \eqref{eq:FiltDirectImageComp} on the direct image complex ${\,'\!}\varphi_+\cO_{S_1\times \Sred}^\beta$ is not necessarily strict).

We will now show that the induced morphism on global sections
$$
\phi:\Gamma(\AA^1_z\times \Sred, M_{\dag +})\longrightarrow \Gamma(\AA^1_z\times \Sred, M)
$$
is filtered, i.e. sends the $\Gamma(\AA^1_z\times \Sred, \cO_{\AA^1_z\times \Sred,})$-submodule
$\Gamma(\AA^1_z\times \Sred,F_k^{H_{sh}} M_{\dag +})$ to $\Gamma(\AA^1_z\times \Sred,F_k M)$. This is obviously equivalent
to the statement of the proposition since $\AA^1_z\times \Sred$ is affine.

We have (denoting $a:\AA^1_z\times \Sred \rightarrow \{pt\}$)
$$
\begin{array}{rcl}
\Gamma(\AA^1_z\times \Sred, M_{\dag +}) &=& a_* M_{\dag +} =
a_* \cH^0 \R ('p_2)_* \text{DR}^{n+\bullet}_{\Pred/\AA_{\lambda_0}^1\times \Sred}(l_{\dag +}\cO_{S_1\times \Sred}^\beta)
\\ \\ &=&
H^0 \R (a\circ {'p}_2)_* \text{DR}^{n+\bullet}_{\Pred/\AA_{\lambda_0}^1\times \Sred}(l_{\dag +}\cO_{S_1\times \Sred}^\beta)
=H^0
\R\Gamma \left(\Pred,\text{DR}^{n+\bullet}_{\Pred/\AA_{\lambda_0}^1\times \Sred}\left(l_{\dag +}\cO_{S_1\times \Sred}^\beta\right)\right)
\end{array}
$$
and similarly
$$
\begin{array}{rcl}
\Gamma(\AA^1_z\times \Sred, M_+) &=& a_* M_+ =
a_* \cH^0 \R ('p_2)_* \text{DR}^{n+\bullet}_{\Pred^*/\AA_{\lambda_0}^1\times \Sred}(l_{1,+}\cO_{S_1\times \Sred}^\beta)
\\ \\ &=&
H^0 \R (a\circ {'p}_2)_* \text{DR}^{n+\bullet}_{\Pred^*/\AA_{\lambda_0}^1\times \Sred}(l_{1,+}\cO_{S_1\times \Sred}^\beta)
=H^0
\R\Gamma \left(\Pred^*,\text{DR}^{n+\bullet}_{\Pred^*/\AA_{\lambda_0}^1\times \Sred}\left(l_{1,+}\cO_{S_1\times \Sred}^\beta\right)\right) \\ \\
&=&H^0\Gamma \left(\Pred^*,\text{DR}^{n+\bullet}_{\Pred^*/\AA_{\lambda_0}^1\times \Sred}\left(l_{1,+}\cO_{S_1\times \Sred}^\beta\right)\right)
\end{array}
$$
(for the last equality, we use that $\Pred^*$ is affine, in contrast to $\Pred$).
Now notice that the morphism $\phi:\Gamma(\AA^1_z\times \Sred, M_{\dag +})\longrightarrow \Gamma(\AA^1_z\times \Sred, M)$
is induced from
\begin{align*}
\R\Gamma \left(\Pred,\text{DR}^{n+\bullet}_{\Pred/\AA_{\lambda_0}^1\times \Sred}\left(l_{\dag +}\cO_{S_1\times \Sred}^\beta\right)\right)
& \ra
\R\Gamma\left(\Pred^*,\text{DR}^{n+\bullet}_{\Pred/\AA_{\lambda_0}^1\times \Sred}\left(l_{\dag +}\cO_{S_1\times \Sred}^\beta\right) \right)\\
& \cong
\R\Gamma\left(\Pred^*,\text{DR}^{n+\bullet}_{\Pred/\AA_{\lambda_0}^1\times \Sred}\left(l_{1,\dag+}\cO_{S_1\times \Sred}^\beta\right) \right)\\
& \cong
\R\Gamma\left(\Pred^*,\text{DR}^{n+\bullet}_{\Pred^*/\AA_{\lambda_0}^1\times \Sred}\left(l_{1,+}\cO_{S_1\times \Sred}^\beta\right) \right)\\
& \cong
\Gamma\left(\Pred^*,\text{DR}^{n+\bullet}_{\Pred^*/\AA_{\lambda_0}^1\times \Sred}\left(l_{1,+}\cO_{S_1\times \Sred}^\beta\right) \right),
\end{align*}
where the first morphism is simply given by restricting sections
of $l_{\dag +}\cO_{S_1\times \Sred}^\beta$ from $\Pred$ to $\Pred^*$ and the isomorphisms are due to the fact that $l$ restricted to $\Pred^*$ is $l_1$, which is proper and affine.

Obviously, the restriction of the Hodge filtration on the module $l_{\dag +}\cO_{S_1\times \Sred}^\beta$ is
the Hodge filtration on the module $l_{1,\dag +}\cO_{S_1\times \Sred}^\beta=l_{1,+}\cO_{S_1\times \Sred}^\beta$.
Since, as has been discussed before, the filtration
$F^H_k \text{DR}^{\bullet+n}_{\Pred/\AA_{\lambda_0}^1\times \Sred}(l_{\dag +}\cO_{S_1\times \Sred}^\beta)$ induces
$F^H_\bullet M_{\dag +}$ and the filtration $F_k^H \text{DR}^{\bullet+n}_{\Pred^*/\AA_{\lambda_0}^1\times \Sred}(l_{1, +}\cO_{S_1\times \Sred}^\beta)$ induces $F_\bullet[-d] M$, we conclude that
$\phi:\Gamma(\AA^1_z\times \Sred, M_{\dag +})\rightarrow \Gamma(\AA^1_z\times \Sred, M)$ sends
$\Gamma(\AA^1_z\times \Sred,F_k^H M_{\dag +})$ to $\Gamma(\AA^1_z\times \Sred,F_k[-d] M)$, or, equivalently,
$\Gamma(\AA^1_z\times \Sred,F_k^{H_{sh}} M_{\dag +})$ to $\Gamma(\AA^1_z\times \Sred,F_k M)$, as required.
\end{proof}

In general, the filtrations on $M_{\dag +}$ and $M$
are not equal, simply because the underlying $\cD_{\AA^1_{\lambda_0} \times \Sred}$-modules are not equal. They become equal after localized partial Fourier transformation as we have shown in Proposition \ref{prop:IsoAfterFourier}. First we will explain that these transformations can be performed at the
filtered level and how the last result can be interpreted in this context.

We will use a general procedure which produces from a filtered $\cD_{\AA^1_{\lambda_0}\times \Sred}$-module $(N,F_\bullet)$ a lattice $\GF_0$ inside $\FL^{\loc}_{\Sred}N$, i.e., an $\cO_{\AA^1_z\times \Sred}$-module which generates $\FL^{\loc}_{\Sred}N$ over $\cD_{\AA^1_z\times \Sred}$.

\begin{defi}\label{def:G0F}(cf. \cite[\S1.d]{Sa8}, \cite[\S A.1]{Sa14})
Let $X$ be a smooth affine variety and let $(N,F_\bullet)$ be a filtered $\cD_{\AA_s^1\times X}$-module, which we identify with its module of global sections. Consider the algebraic microlocalization
$$N[\dd_s^{-1}]:=\CC[s]\langle\dd_s,\dd_s^{-1}\rangle\otimes_{\CC[s]\langle\dd_s\rangle}N.$$
By letting act $\tau$ as $-\dd_s$ and $\dd_\tau$ as $s$, we consider $N[\dd_s^{-1}]$ as a $\cD_X[\tau,\tau^{-1}]\langle\dd_\tau \rangle$-module (which actually coincides with $\FL^{\loc}_X N$). Let now $\whloc$ be the natural localization morphism $\whloc:N\ra N[\dd_s^{-1}]$. Then we define
\begin{equation}\label{eq:DefBrieskorn}
\GF_0\FL^{\loc}_X N:=\sum_{j\geq 0}\dd_s^{-j}\whloc(F_j N),
\end{equation}
notice that then $\GF_0\FL^{\loc}_X N$ has naturally the structure of a $\cRint_{\AA_z^1\times X}$-module. We also put
for any $k\in\ZZ$
$$\GF_k\FL^{\loc}_X N:=z^k\cdot \GF_0\FL^{\loc}_X N=\sum_{j\geq 0}\dd_s^{-(j+k)}\whloc(F_j N)=\sum_{j\geq 0}\dd_s^{-j}\whloc(F_{j+k} N).$$
\end{defi}

There is an interpretation of this construction as a Fourier-Laplace transformation
for $\cRint_{\AA_z^1\times\AA^1_s\times X}$-modules as explained in \cite[Rem. A.3]{Sa14}. Using this interpretation, one can show the following fact.
\begin{lemma}\label{lem:MHMIntoIrrMHMbyFL}
Let $(N,F_\bullet)$ be a filtered $\cD_{\AA_s^1\times X}$-module underlying an element in $\MHM(\AA^1_s\times X, \CC)$ (the abelian category of complex mixed Hodge modules). Then the $\cRint_{\AA^1_z\times X}$-module $G_0^{F_\bullet} \FL^{\loc}_X(N)$
underlies an element of $\IrrMHM(X)$.
\end{lemma}
\begin{proof}
We first define the algebraic Fourier-Laplace transformation for integrable $\cR$-modules as in \cite[Defs. 3.2, 3.7]{SevCastReich}. Consider the following diagram:
$$
\begin{tikzcd}
\AA^1_s\times X \ar[hook]{rr}{j} \ar{ddrr}{q}&& \PP^1_s\times X \ar{dd}{\overline{q}} \\ \\
&&X
\end{tikzcd},
$$
where $j$ is the canonical open embedding and $q$ and $\bar{q}$ the respective second projections. Let $\cA^{s/z}_{\textup{aff}}$ the $\cR_{\AA^1_z \times \AA_s  \times X}$-module $\cO_{\AA^1_z \times \AA_s  \times X}$ equipped with the $z$-connection $zd+ds$. Then for an algebraic $\cRint_{\AA^1_z\times \AA^1_s\times X}$-module $\cN$,
we put
$$
\FL^{\cR}_X(\cN):=\cH^0 q_+\left(\cN\otimes \cA^{s/z}_{\textup{aff}}\right)
 \in \textup{Mod}(\cRint_{\AA^1_z\times X}).
$$
Let $D$ be the reduced divisor $(\PP^1_s\backslash \AA^1_s)\times X = \{\infty\}\times X$. Then $\cA^{s/z}_*:= j_* \cA^{s/z}_{\textup{aff}}$ carries a natural structure of an $\cR_{\AA^1_z \times \PP^1_s \times X}(*D)$-module. Let $\cE_*^{s/z}$ be the analytification of $\cA^{s/z}_*$.

As in the case of $\cD$-modules, there is the notion of strict specializability and $V$-filtration for $\cRint$-modules; check \cite[\S 2.1.2]{Mo13} for more details.

Notice that  the meromorphic function $s\in \cO_{\PP^1_s\times X}(*D)$ evidently extends to a map $\PP^1_s\times X \rightarrow \PP^1_s$ whose reduced pole divisor is exactly $D\cong X$, in particular, it is smooth. Then it follows from \cite[Lem. 3.1]{Sa14} (using that, according to their notation, $P_{\textup{red}}=D$ and $\mathbf{e}=1$) that $\cE_*^{s/z}$ is coherent not only over $\cR_{\AA^1_z \times \PP^1_s  \times X}$ but over $V_0\cR_{\AA^1_z \times \PP^1_s  \times X}$ too. By \cite[\S 2.1.2]{Mo13} we know that $\cE_*^{s/z}$ is automatically specializable along $D$ and its corresponding $V$-filtration is trivial. Therefore, by construction, $\cE^{s/z}:=\cE_*^{s/z}[*D]=\cE_*^{s/z}$ (cf. [ibid., \S 3.1.2]). We then know by \cite[Prop. 3.3]{Sa14} that $\cE^{s/z}$ underlies an algebraic, integrable
pure twistor $\cD$-module $\cT^{s/z}$ on $\AA^1_s\times X$. (Cf. \cite[\S 1.6.a]{Sa15} for another discussion on this issue.)

Now suppose that $\cN$ underlies an algebraic, integrable mixed twistor $\cD$-module $\mathscr{N}=(\cN',\cN,C)\in\MTM_{\textup{alg}}^{\textup{int}}(\AA^1_s\times X)$, then we can define
its relative Fourier-Laplace transform
$$
\FL^{\textup{MTM}}_X(\mathscr{N}) := \cH^0 q_* (\mathscr{N}\otimes \cT^{s/z})
\in \MTM_{\textup{alg}}^{\textup{int}}(X).
$$
For notational convenience, for any $\cR$-triple $\mathscr{K}=(\cK',\cK,C)$ we define the forgetful
functor $\textup{For}(\mathscr{K}):=\cK$. Then we have a comparison formula, the proof of which is completely
analogous to \cite[Prop. 3.5]{SevCastReich}:

Let $\mathscr{N}$ be an element in $\MTM_{\textup{alg}}^{\textup{int}}(\AA^1_s\times X)$, then there is an isomorphism
of $\cR_{\AA^1_z\times X}$-modules:
\begin{equation}\label{eq:FLCompar}
\textup{For}(\FL^{\MTM}_X(\mathscr{N})) \cong z^{-1}\FL^{\cR}_X(\textup{For}(\mathscr{N})).
\end{equation}
In other words, if $\cN$ is an $\cR_{\AA^1_z\times \AA^1_s\times X}$-module underlying an algebraic, integrable mixed twistor $\cD$-module on $\AA^1_s\times X$, then
the (shifted) relative algebraic Fourier-Laplace transform $z^{-1}\FL^{\cR}_X(\cN)$ underlies an element of $\MTM_{\textup{alg}}^{\textup{int}}(X)$.

Now let us suppose that we are given a filtered $\cD_{\AA^1_s\times X}$-module $(N,F_\bullet)$
underlying a complex mixed Hodge module on $\AA^1_s\times X$. Then its Rees module
$\cR_F \cN$ is an algebraic $\cR_{\AA^1_z\times \AA^1_s\times X}$-module and underlies an element $\mathscr{N}$ in $\MTM_{\text{alg}}^{\text{int}}(\AA^1_s\times X)$ as well as in $\IrrMHM(\AA^1_s\times X)$. Since all the functors
entering in the definition of $\FL_X^{\MTM}$ preserve the category of irregular Hodge modules by \cite[Cor. 0.5.]{Sa15},
we know by the above comparison result (i.e. by formula \eqref{eq:FLCompar}) that $z^{-1}\FL^{\cR}_X(\cR^{F} N)$ underlies an irregular Hodge module on $X$. Now, by \cite[Rem. A.3]{Sa14} (noting that the argument there is valid for the partial transformation as well), we can identify $\FL^{\cR}_X(\cR^{F} N)$ with $zG_0^{F_\bullet} \FL^{\loc}_X(N)$, so that
$$
G_0^{F_\bullet} \FL^{\loc}_X(N)\cong z^{-1}\FL^{\cR}_X(\cR^{F} N).
$$
This shows that $G_0^{F_\bullet} \FL^{\loc}_X(N)$ underlies an irregular Hodge module on $X$,
as required.

\end{proof}

With these definitions at hand, we have the following consequence of Lemma \ref{lem:InclusionHodgeFilt}:
\begin{coro}
In the above situation, we have
$$
G^{F^{H_{sh}}_\bullet}_0 \FL^{\loc}_{\Sred} M_{\dag +}\subset \GF_0 \FL^{\loc}_{\Sred} M
$$
\end{coro}
\begin{proof}
This is a direct consequence of the definition in formula \eqref{eq:DefBrieskorn}, taking into account
Lemma \ref{lem:InclusionHodgeFilt} and the fact that the filtered morphism $\phi$ induces an isomorphism
of $\cD_{\AA^1_z\times \Sred}$-modules by applying the functor $\FL^{\loc}_{\Sred}$.
\end{proof}

From now on, we will specify the above situation to our main example, where the matrix $A$ is given by
$$
A=
\begin{pmatrix}
  1 & -1 & 0 & 0 &\ldots  & 0  \\
  1 &  0 & -1 & 0 & \ldots & 0 \\
   \vdots &\ldots \\
  1 &  0 & 0 & 0 & \ldots & -1
\end{pmatrix}.
$$
In particular, we have $d=n-1$, and $\Sred = \GG_{m,t}$.
We still write $M_{\dag +}=\cH^0 p_{2,+} l_{\dag\,+}\cO^\beta_{S_1\times \GG_{m,t}}$ and $M=\cH^0{\,'\!}\varphi_+ \cO^\beta_{S_1\times \GG_{m,t}}$.
We seek
to improve the inclusion of the last corollary to an equality of lattices
inside the $\cD_{\AA^1_z \times \GG_{m,t}}$-modules
$\FL^{\loc}_{\gmt} (M_{\dag +})
\cong
\FL^{\loc}_{\gmt} (M)$.
We will follow an argument from the proof of \cite[Lem. 4.7]{Sa8}. In order to do so, we have to make more explicit
the structure of the module $\FL^{\loc}_{\gmt}(M)\cong\FL^{\loc}_{\gmt}(M_{\dag\,+})$,
which is done by the next lemma.
\begin{lemma}
\begin{enumerate}
\item
The singular locus
$\Sigma:=\text{Sing}(M)=\text{Sing}(\cH^0\varphi_+ \cO^\beta_{S_1\times \GG_{m,t}})$
is given by
$$
\Sigma=\bigcup_{\xi\in\mu_n} \{(n\cdot \xi\cdot t',t)\} \subset \AA^1_{\lambda_0} \times \GG_{m,t},
$$
where $t'$ is an $n$-th root of $t$, chosen without loss of generality.
\item
Write $D:=\{0\}\times \GG_{m,t}\subset \AA^1_z\times \GG_{m,t}$. Then
$\FL^{\loc}_{\gmt} (M)$ is
$\cO_{\AA^1_z\times \GG_{m,t}}(*D)$-locally free of rank $n$.
\item
Consider the sheaf
$\widehat{\cO}_{\AA^1_z\times \GG_{m,t}}$, which is the formal completion of $\cO_{\AA^1_z\times \GG_{m,t}}$ along
the divisor $\{0\}\times \GG_{m,t}$. Then we have a decomposition (as sheaves on $\{0\}\times\GG_{m,t}$)
$$
\FL^{\loc}_{\gmt}(M)
\otimes_{\cO_{\AA^1_z\times \GG_{m,t}}(*D)} \widehat{\cO}_{\AA^1_z\times \GG_{m,t}}(*D)
\cong \bigoplus_{\xi\in\mu_n} \widehat{\cE}_\xi \otimes \widehat{\cN}_{\alpha_\xi},
$$
where $\widehat{\cE}_\xi:=(\widehat{\cO}_{\AA^1_z\times \GG_{m,t}}(*D), d-d(n \cdot \xi \cdot  t'/z))$ and
$\widehat{\cN}_{\alpha_\xi}=(\widehat{\cO}_{\AA^1_z\times \GG_{m,t}}(*D), d+\alpha_\xi dz/z)$, with $\alpha_\xi\in \CC$.
\end{enumerate}
\end{lemma}
\begin{proof}
\begin{enumerate}
  \item
  It is well known that the singular locus of a Gauss-Manin system, i.e., of the top-cohomology
  of the direct image complex $\varphi_+ \cM$, is nothing but the discriminant of the morphism $\varphi$
  provided that the module $\cM$ is smooth (which is the case here, since $\cM=\cO^\beta_{S_1\times \GG_{m,t}})$.
  Recall (see formula \eqref{eq:LaurPolPn}) that ${\,'\!}\varphi(y_1,\ldots,y_{n-1},t)=\left(\frac{1}{y_1}+\ldots+\frac{1}{y_{n-1}}+t\cdot y_1\cdot\ldots\cdot y_{n-1}, t \right)$.
  One easily checks (see, e.g., \cite[\S 1.B.]{DS2}) that a point $(y_1,\ldots,y_{n-1})\in S_1$ is critical if and only if $y_1=\ldots=y_{n-1}=:y$ and $y^n\cdot t=1$ (and that all critical points are Morse). Then the critical values are as indicated.
  \item This is a direct consequence of the second point of \cite[Thm. 1.11]{DS}, since by the discussion above the singular locus $\Sigma$ satisfies the assumption (NC) of loc. cit. It is also known that the rank of
  $\FL^{\loc}_{\gmt}(M)$
equals the global Milnor number
  of $\varphi$, i.e., the numbers of critical points, which is $n$.
  \item
  This follows from \cite[Ch. III, Thm. 5.7]{Sa4}, since the critical values of $\varphi$, i.e., the eigenvalues
  of the pole part of $z^2\nabla_z$, are distinct for any $t\in \GG_{m,t}$.
\end{enumerate}
\end{proof}

With these preparations, we can state the next result. As has been explained at the end of section \ref{sec:reduction} from page \pageref{page:Roadmap} on, it is the main step to show that the $\cRint_{\AA^1_z\times \gmt}$-module $\wh{\cH}$ underlies
a mixed Hodge module (that is, the content of Theorem \ref{thm:HypIrrMHM}). The explicit description of $\wh{\cH}$ as a cyclic quotient by two operators can be rather easily identified with the object $\GF_0 \FL^{\loc}_{\gmt} M$, as we will see below in the proof of Theorem \ref{thm:HypIrrMHM}, whereas the Hodge theoretic property we want (i.e., the fact that
it is an object in $\IrrMHM(\gmt)$) holds for $G^{F^{H_{sh}}_\bullet}_0
\FL^{\loc}_{\gmt} M_{\dag +}$. Hence we need to identify these two
$\cRint_{\AA^1_z\times\gmt}$-modules.
\begin{thm}\label{thm:EqualFourierFilt}
In the above situation, we have
$$
G^{F^{H_{sh}}_\bullet}_0
\FL^{\loc}_{\gmt} M_{\dag +} =
\GF_0 \FL^{\loc}_{\gmt} M.
$$
\end{thm}
\begin{proof}
We have already proved the inclusion
$$
G^{F^{H_{sh}}_\bullet}_0
\FL^{\loc}_{\gmt} M_{\dag +} \subset
\GF_0 \FL^{\loc}_{\gmt} M
$$
of $\cO_{\AA^1_z\times \GG_{m,t}}$-modules. Since both sheaves coincide outside the divisor
$D=\{0\}\times \GG_{m,t}$, and since
$\widehat{\cO}_{\AA^1_z\times \GG_{m,t}}$ is $\cO_{\left(\AA^1_z\times \GG_{m,t},D\right)}$-flat,
it is therefore sufficient to show that
$$
G^{F^{H_{sh}}_\bullet}_0 \FL^{\loc}_{\gmt} M_{\dag +} \otimes \widehat{\cO}_{\AA^1_z\times \GG_{m,t}}
=
\GF_0 \FL^{\loc}_{\gmt} M
\otimes \widehat{\cO}_{\AA^1_z\times \GG_{m,t}}.
$$
This follows as in the proof of \cite[Lem. 4.7]{Sa8}: Using the formal
decomposition result from the last Lemma, both modules can be interpreted as microlocal filtered direct images under ${{'\!}p}_2$
of two modules which coincide on $\Pred^*$. In these direct images, the contributions from $\Pred\backslash \Pred^*$
vanish by the last statement of Proposition \ref{prop:IsoAfterFourierRed} and Lemma \ref{lem:StratSmooth} (notice that inside $P=\PP^n \times \AA^1_{\lambda_0} \times S_2$,
we have $\Pred^*\cap \Gamma X \cong S_1 \times \Sred$ where we see $\Sred$ as a subspace of $S_2$ via the embedding $\iota$), and therefore both modules are equal.
\end{proof}

We are finally able to complete the proof of Theorem \ref{thm:HypIrrMHM}. It remains
to show that the $\cRint_{\AA^1_z\times\gmt}$-module
$$
\wh{\cH} = \frac{\cRint_{\AA^1_z\times\gmt}}{\left( z^2\partial_z+ntz\partial_t-z\sum_{i=1}^n \alpha_i), \prod_{i=1}^n z(t\partial_t-\alpha_i)-t
\right)}
$$
associated to
the purely irregular hypergeometric $\cD_{\gmt}$-module $\cH(\alpha_i,\emptyset)$ underlies
an object of the category $\IrrMHM(\gmt)$.

\label{page:MainProof}
\begin{proof}[End of the proof of theorem \ref{thm:HypIrrMHM}]
Let us assume first that $\alpha_1=0$, denote by $\alpha$ the vector $(\alpha_2,\ldots,\alpha_n)$ and put $\alpha_0:=0$.
From Lemma \ref{lem:RhypGKZ} and Proposition \ref{prop:TwdeRham} we conclude that
\begin{align*}
  \wh\cH& \cong \iota^+ \cH^0\left(\pi_{2,*}\Omega_{S_1\times S_2/S_2}^{\bullet+d}[z],z\left(d-\kappa(\alpha)\wedge\right)-df\wedge\right) \\
  & \cong \cH^0\left(\pi_{2,*}\Omega_{S_1\times \gmt/\gmt}^{\bullet+d}[z],z\left(d-\kappa(\alpha)\wedge\right)-d{\,'\!}f\wedge\right),
\end{align*}
recall that  ${\,'\!}f$ is the first component of ${\,'\!}\varphi$, as written in (\ref{eq:LaurPolPn}). Finally, it is easy to see from Definition \ref{def:G0F} that we have
$$
\cH^0\left(\pi_{2,*}\Omega_{S_1\times \gmt/\gmt}^{\bullet+d}[z],z\left(d-\kappa(\alpha)\wedge\right)-d{\,'\!}f\wedge\right)
\cong
\GF_0 \FL^{\loc}_{\gmt} \cH^0 {\,'\!}\varphi_+ \cO^\alpha_{S_1\times \gmt}
=
\GF_0 \FL^{\loc}_{\gmt} M.
$$

Since by Theorem \ref{thm:EqualFourierFilt} we can further conclude
$$
\wh\cH\cong G^{F^{H_{sh}}_\bullet}_0 \FL^{\loc}_{\gmt}M_{\dag +},
$$
we obtain that $\wh\cH$ underlies an element of $\IrrMHM(\GG_{m,t})$ by Lemma \ref{lem:MHMIntoIrrMHMbyFL} 
(recall that $M_{\dag +}$ underlies a pure polarizable complex Hodge module).
Restricting $\wh\cH$ to $z=1$ we get the original $\cD_{\GG_{m,t}}$-module $\cH(\alpha_i;\emptyset)$.

Assume now that $\alpha_1\neq0$. The tensor product of $\cRint_{\AA_z^1\times \GG_{m,t}}$-modules $\wh\cH\otimes_{\cO_{\AA_z^1\times\GG_{m,t}}}\wh\cK_{-\alpha_1}$ gives rise to the corresponding tensor product of twistor $\cD$-modules on $\GG_{m,t}$. This product can be presented as $\wh\cH(\alpha'_i;\emptyset)$, where $\alpha'_i=\alpha_i-\alpha_1$ for every $i$. Reasoning as above, since $\alpha'_1=0$, such tensor product is an irregular mixed Hodge module. Since $\wh\cK_{\alpha_1}$ is the faithful image of a mixed Hodge module on $\GG_{m,t}$, the tensor product with it preserves the condition of being in $\IrrMHM(\GG_{m,t})$ due to \cite[Cor. 0.5]{Sa15}, and so is the case of our original $\cRint_{\AA_z^1\times \GG_{m,t}}$-module
$$\wh\cH\cong\wh\cH(\alpha'_i;\emptyset)\otimes_{\cO_{\AA_z^1\times\GG_{m,t}}}\wh\cK_{\alpha_1}.$$
\end{proof}

\section{The irregular Hodge filtration}

In this section we will prove the second main result of this paper. Let us recall the notations used above. For a positive integer number $n$, and $\alpha_1,\ldots,\alpha_n$ real numbers, we consider the hypergeometric $\cD_{\GG_{m,t}}$-module $\cH=\cH(\alpha_i;\emptyset)$ and its associated twistor $\cD$-module $\wh\cH$ on $\GG_{m,t}$ (we will denote by the same symbol the underlying hypergeometric $\cRint_{\AA_z^1\times\GG_{m,t}}$-module). From \cite[Thm. 0.7]{Sa15} and Theorem \ref{thm:HypIrrMHM} we know that there exists a unique irregular Hodge filtration of $\cH$. We provide it in Theorem \ref{thm:HodgeData} below.

Let $\wh\cM$ a twistor $\cD$-module on $X$, and call its associated $\cR_{\AA_z^1\times X}$-module the same way. If the $\cR$-module $\wh\cM$ is integrable, good and well-rescalable (\hspace{-.5pt}\cite[Def. 2.19]{Sa15}), we can define the irregular Hodge filtration of the underlying $\cD_X$-module (say $\cM$) following [ibid., Def. 2.22].

In our particular context, let us use the following notation (cf. \cite[Not. 2.1]{Sa15}) for the sake of brevity: We will write $\cX:=\AA_z^1\times \GG_{m,t}$, $\ttheta\cX=\cX\times\mathds{G}_{m,\theta}$, $\ttau\cX=\cX\times\AA_\tau^1$ and $\ttau \cX_0=\cX\times\{\tau=0\}$, where $\theta=1/\tau$.

Let us summarize the process we follow to achieve our goal. We must first consider the rescaling of $\wh\cH$: this is the inverse image $\ttheta\wh\cH:=\mu^*\cH$ (as $\cO_{\ttheta\cX}$-module), endowed with a natural action of $\cRint_{\ttheta\cX}$ as depicted in \cite[2.4]{Sa15} (note that $\theta=\tau^{-1}$), where $\mu$ is the morphism given in [ibid., Not. 2.1] by
$$\begin{array}{rrcl}
\mu:&\ttheta\cX&\ra&\cX\\
&(z,t,\theta)&\mapsto&(z\theta,t).
\end{array}$$
This is done right below. Then we have to invert $\theta$ to obtain an $\cRint_{\ttau\cX}(*\ttau \cX_0)$-module $\ttau\wh\cH$, to work in the context of \cite[\S 2.3]{Sa15}. Finally, the irregular Hodge filtration is obtained from a suitable $V$-filtration along the divisor $\tau=0$ defined on $\ttau\wh\cH$, which is called $\ttau V$-filtration (the new symbol $\ttau V$ is to make clear the variety over which we are working; note the same convention in [ibid.], from Remark 2.20 on). We will actually \emph{define} a filtration on $\ttau\wh\cH$ in Definition \ref{defi:tauVfilt}, and then prove that it equals the $\ttau V$-filtration in Proposition \ref{prop:tauVfiltr}, following \cite[\S 2.1.2]{Mo13}.

\begin{prop}\label{prop:rescaling}
Recall that we could write $\wh\cH$ as the $\cRint_{\cX}$-module $\cRint_{\cX}/(P,H)$, where $P$ and $H$ were, respectively,
$$z^2\dd_z+nzt\dd_t+\gamma z\,\text{ and }\,\prod_{i=1}^nz(t\dd_t-\alpha_i)-t,$$
for certain value of $\gamma$. Then, $\ttheta\wh\cH=\cRint_{\ttheta\cX}/(P,\ttheta R,\ttheta H)$, with $P$ as before and
$$\ttheta R=z^2\dd_z-z\theta\dd_\theta\,\text{ and }\,\ttheta H=\prod_{i=1}^nz\theta(t\dd_t-\alpha_i)-t.$$
\end{prop}
\begin{proof}
The morphism $\mu$ can be decomposed as $p\circ\phi$, where $p$ is the canonical projection from $\ttheta\cX$ to $\cX$ and $\phi$ is the automorphism of $\ttheta\cX$ given by $(z,t,\theta)\mapsto(z\theta,t,\theta)$. Then, in the category of $\cO_{\ttheta\cX}$-modules we have that
$$\mu^*\wh\cH\cong\phi^*p^*\wh\cH\cong\phi^*\cRint_{\ttheta\cX}/(z\theta\dd_\theta, P,H(z,t,\dd_t))\cong\cRint_{\ttheta\cX}/(z^2\dd_z-z\theta\dd_\theta,P,H(z\theta,t,\dd_t)),$$
by the chain rule.

What remains now is to prove the compatibilities of \cite[2.4]{Sa15} among the actions of $\cRint_{\ttheta\cX}$ on $\cRint_{\ttheta\cX}/(P,\ttheta R,\ttheta H)$, seen as $\mu^*\wh\cH=\mu^{-1}\wh\cH\otimes_{\mu^{-1}\cO_{\cX}}\cO_{\ttheta\cX}$.
They are just a consequence of the presence of $z^2\dd_z-z\theta\dd_\theta$ in the ideal with which we take the quotient in $\mu^*\wh\cH$ and how $z$, $z\dd_i$ or $z^2\dd_z$ act on both factors of the tensor product. For instance, if we multiply by $z$ at the right-hand one is the same as if we multiply by $z\theta$ at the left-hand one.
\end{proof}

\begin{rem}\label{rem:chi}
Let $i_{\tau=z}$ be the inclusion $\mathds{G}_{m,z}\times\GG_{m,t}\hra\ttheta\cX$ given by $(z,t)\mapsto(z,t,\tau)$. Note that, according to the fourth point of \cite[Lem. 2.5]{Sa15}, we must have $i_{\tau=z}^*\ttheta\wh\cH\cong\pi^{0,+}\cH$ as $\cR_{\mathds{G}_{m,z}\times\GG_{m,t}}$-modules, with $\pi^0$ being the projection $\mathds{G}_{m,z}\times\GG_{m,t}\ra \GG_{m,t}$. Indeed, we have
\begin{align*}
i_{\tau=z}^*\ttheta\wh\cH = \cRint_{\ttheta\cX}/(P,\ttheta R,\ttheta H,\theta z-1)
& \cong
\cO_{\mathds{G}_{m,\theta}\times \AA^1_z\times \GG_{m,t}}\langle zt\dd_t\rangle/(H_1,\theta z-1) \\ \\
& \cong
\cO_{\mathds{G}_{m,\theta}\times \mathds{G}_{m,z} \times \GG_{m,t}}\langle zt\dd_t\rangle/(H_1,\theta -z^{-1}) \\ \\
& \cong
\cO_{\mathds{G}_{m,z}\times \GG_{m,t}}\langle zt\dd_t\rangle/(H_1) \\ \\
& \cong \cO_{\mathds{G}_{m,z}\times \GG_{m,t}} \otimes_{\cO_{\GG_{m,t}}} \cH  \cong\pi^{0,+}\cH,
\end{align*}
where $H_1$ is the result of replacing $z$ by 1 at the expression for $H$.
\end{rem}

As said above, in order to continue, we must pass from $\ttheta\cX$ to $\ttau\cX$. Therefore, we invert $\theta$ and extend $\tau$ to the affine line to get a $\cRint_{\ttau\cX}(*\ttau \cX_0)$-module. In other words, call $\inv:\mathds{G}_{m,\theta}\rightarrow\mathds{G}_{m,\tau}$ the inversion operator $\theta\mapsto\theta^{-1}=\tau$ and $j:\mathds{G}_{m,\tau}\hookrightarrow\AA_{\tau}^1$ the canonical inclusion. From now on, we will denote by $\ttau\wh\cH$ the $\cRint_{\ttau\cX}(*\ttau \cX_0)$-module $(\id_{\cX}\times(j\circ\inv))_*\ttheta\wh\cH$. By virtue of Proposition \ref{prop:rescaling} we can write $\ttau\wh\cH$ as the $\cRint_{\ttau\cX}(*\ttau \cX_0)$-module $\ttau\wh\cH=\cRint_{\ttau\cX}(*\ttau \cX_0)/(P,\ttau R,\ttau H)$, with $P$ as before and
$$\ttau R=z^2\dd_z+z\tau\dd_\tau\,\text{ and }\,\ttau H=\prod_{i=1}^n\frac{z}{\tau}(t\dd_t-\alpha_i)-t.$$

\begin{lemma}\label{lem:ConnHypResc}
For each $k=0,\ldots,n-1$, let $Q_k$ be the operator
$$Q_k=(-n)^k\prod_{j=1}^k\frac{z}{\tau}(t\dd_t-\alpha_j),$$
where the empty product must be understood as one. Then the $Q_k$ form a basis of $\ttau\wh\cH$ as an $\cO_{\ttau\cX}(*\ttau \cX_0)$-module. The integrable connection arising from the $\cRint_{\ttau\cX}(*\ttau \cX_0)$-module structure associated with $\ttau\wh\cH$ has the following matrix expression with respect to that basis:
$$\nabla\underline Q=\underline Q\left(\left(\tau A_0+zA_\infty\right)\frac{dz}{z^2}+\left(-\tau A_0+zA'_\infty\right)\frac{dt}{nzt}-\left(\tau A_0+zA_\infty\right)\frac{d\tau}{z\tau}\right),$$
where $A_0$, $A'_\infty$ and $A_\infty$ are the matrices
$$A_0=\left(\begin{array}{cccc}
0& & & (-n)^nt\\
1&\ddots& & 0\\
 &\ddots&0&\vdots\\
 & & 1& 0\end{array}\right),\,A'_\infty=\diag(n\alpha_1,\ldots,n\alpha_n)$$
$$\text{\emph{and }}\,A_\infty=\diag(0,1,\ldots,n-1)-\gamma I_n-A'_\infty.$$
\end{lemma}
\begin{proof}
We can use the expressions for $\ttau R$ and $P$ to replace the classes of $z\tau\dd_\tau$ and $z^2\dd_z$, respectively, in terms of $zt\dd_t$. If we extend for a moment the definition of the $Q_k$ for values of $k$ greater than $n-1$ taking $Q_k:=\left(z\tau^{-1}t\dd_t\right)^{k-n+1}Q_{n-1}$, we see that $\ttau\wh\cH$ is generated as a $\cO_{\ttau\cX}(*\ttau \cX_0)$-module by the $Q_k$, for $k\geq0$ (since it is obviously generated by the powers of $zt\dd_t$, and those can be expressed by the $Q_k$). If we focus now at the degree in $zt\dd_t$ of the generators, we can use $\ttau H$ to get rid of any $Q_k$ with $k\geq n$. Since $\deg_{zt\dd_t}Q_k=k$, they must be linearly independent over $\cO_{\ttau\cX}(*\ttau \cX_0)$ and so they are a basis.

Let now $k<n-1$. Then from the relation $-nz/\tau(t\dd_t-\alpha_{k+1})Q_k=Q_{k+1}$ we can write that $nzt\dd_tQ_k=-\tau Q_{k+1}+nz\alpha_{k+1}Q_k$. Now if $k=n-1$, then
$$-nz/\tau(t\dd_t-\alpha_n)Q_{n-1}=(-n)^n\prod_{j=1}^nz/\tau(t\dd_t-\alpha_j)=(-n)^nt.$$
This gives us the second summand of the formula above in the statement.

The first one is a consequence of the last one and the fact that the class of $z^2\dd_z+z\tau\dd_\tau$ vanishes in $\ttau\wh\cH$; let us show the expression for the latter.

Take again $k<n-1$. Then,
$$z\tau\dd_\tau Q_k=(-n)^kz\tau\dd_\tau \tau^{-k}\prod_{j=1}^kz(t\dd_t-\alpha_j)=(-n)^kz(-k\tau^{-k}+\tau^{-k+1}\dd_\tau) \prod_{j=1}^kz(t\dd_t-\alpha_j)=$$
$$=-kzQ_k+Q_kz\tau\dd_\tau=-kzQ_k+Q_kz(nt\dd_t+\gamma)=z(nt\dd_t+\gamma-k)Q_k=-\tau Q_{k+1}+z(n\alpha_{k+1}+\gamma-k)Q_k.$$
The analogous calculation for $k=n-1$ gives us that $z\tau\dd_\tau Q_{n-1}=-\tau(-n)^nt+z(n\alpha_n+\gamma-(n-1))Q_{n-1}$.
\end{proof}

\begin{defi}\label{defi:tauVfilt}
For each $\alpha\in\RR$, let us define the following subsets of $\ttau\wh\cH$:
$$\ttau U_\alpha\ttau\wh\cH:=\left\{\sum_{k=0}^{n-1}f_k\tau^{\nu_{k}}Q_k\,:\, f_k\in\cO_{\ttau\cX}\text{ , }\,\max(k-n\alpha_{k+1}-\gamma-\nu_k)\leq\alpha\right\},$$
$$\ttau U_{<\alpha}\ttau\wh\cH:=\left\{\sum_{k=0}^{n-1}f_k\tau^{\nu_k}Q_k\,:\, f_k\in\cO_{\ttau\cX}\text{ , }\,\max(k-n\alpha_{k+1}-\gamma-\nu_k)<\alpha\right\}.$$
\end{defi}

\begin{rem}\label{rem:tauVfilt}
Note that the $\ttau U_\alpha\ttau\wh\cH$ form an increasing filtration of $\ttau\wh\cH$, indexed by the real numbers but with a discrete set of jumping numbers, such that $\tau\ttau U_\alpha\ttau\wh\cH=\ttau U_{\alpha-1}\ttau\wh\cH$ for any $\alpha$. The graded piece associated with $\alpha$ is $\Gr_\alpha^{\ttau U}\ttau\wh\cH=\ttau U_\alpha\ttau\wh\cH/\ttau U_{<\alpha}\ttau\wh\cH$.

In the definition of the $\ttau U_\alpha\ttau\wh\cH$ all the exponents $\nu_k$ of the powers of $\tau$ accompanying the $f_kQ_k$ satisfy that $\nu_k\geq-\alpha+k-n\alpha_{k+1}-\gamma$. Therefore, we can define the steps of the filtration in an alternative way, as the free $\cO_{\ttau\cX}$-modules of finite rank
$$\ttau U_\alpha\ttau\wh\cH=\bigoplus_{k=0}^{n-1}\cO_{\ttau\cX}\cdot\tau^{\nu_{\alpha}(k)}Q_k,$$
where $\nu_\alpha(k)=\lceil-\alpha+k-\gamma-n\alpha_{k+1}\rceil$.

With this expression it is easy to see that $\ttau U_\alpha\ttau\wh\cH/(\tau-z)\ttau U_\alpha\ttau\wh\cH$ is the $z$-graded free $\cO_{\cX}$-module $\bigoplus_k\cO_{\cX}z^{\nu_\alpha(k)}\bar{Q}_k$, where
$$\bar{Q}_k=(-n)^k\prod_{j=1}^k(t\dd_t-\alpha_j),$$
and that the graded pieces $\text{Gr}_\alpha^{\ttau U}\ttau\wh\cH$ are
$$\text{Gr}_\alpha^{\ttau U}\ttau\wh\cH=\bigoplus_{k=0}^{n-1}\cO_{\cX}\cdot\tau^{\nu_{\alpha}(k)}Q_k,$$
which are strict $\cR_{\cX}$-modules.
\end{rem}

Recall that as in the case of $\cD$-modules, we have a notion of strict specializability and $V$-filtration for $\cRint$-modules. In the setting under consideration we recall that we will use $\ttau V$ to denote the canonical $V$-filtration of a $\cRint_{\ttau\cX}$-module. We recall as well the reference \cite[\S 2.1.2]{Mo13} for more details.

\begin{prop}\label{prop:tauVfiltr}
Assume the $\alpha_i$ lie in the interval $[0,1)$, increasingly ordered. Then, $\ttau\wh\cH$ is strictly $\RR$-specializable along $\ttau \cX_0$ and its $\ttau V$-filtration is in fact given by the $\ttau U_\alpha\ttau\wh\cH$.
\end{prop}
\begin{proof}
First of all we will see that $\ttau U_\alpha\ttau\wh\cH$ is the $\ttau V$-filtration of $\ttau\wh\cH$, following \cite[\S\S 2.1.2.1, 2.1.2.2]{Mo13}. Apart from what we already shown at the remark above, what remains then is showing conditions iii' and v of [ibid., \S 2.1.2] and prove that the $\ttau U_\alpha\ttau\wh\cH$ are coherent $V_0\cR_{\cX}$-modules. We will start by the second condition. We will consider then the mappings $\fp,\fe$ given by
$$\begin{array}{rcl}
(\fp,\fe):\RR\times\CC&\longrightarrow&\RR\times\CC\\
(\beta,\omega)&\longmapsto&(\beta+2\Re(z\bar\omega),-\beta z+\omega-\bar{\omega}z^2)\end{array}.$$
We ought to see now that the operator $z\tau\dd_\tau-\fe(\beta,\omega)$ is nilpotent on the graded pieces $\Gr_\alpha^{\ttau U}\ttau\wh\cH$ only for a finite amount of $(\beta,\omega)\in\cK:=\{\beta+2\Re(z_0\bar\omega)=\alpha\}$, for any value $z_0$ of $z$. Moreover, $\ttau\wh\cH$ will be strictly $\RR$-specializable if those $(\beta,\omega)$ belong in fact to $\RR\times\{0\}$ (cf. \cite[\S 1.3.a]{Sa15}).

Take then $(\beta,\omega)\in\cK$ and $f\tau^\nu Q_k\in\ttau U_\alpha\ttau\wh\cH$, with $f\in\cO_{\ttau\cX}$. We must have that $k-n\alpha_{k+1}-\gamma-\nu\leq\alpha$. Assume that $k<n-1$. Thanks to Lemma \ref{lem:ConnHypResc} we know that
$$(z\tau\dd_\tau-\fe(\beta,\omega))f\tau^\nu Q_k=\big(z\tau\dd_\tau+(\nu+n\alpha_{k+1}+\gamma-k+\beta)z-\omega+\bar{\omega}z^2\big)(f)\tau^\nu Q_k-\tau^{\nu+1}Q_{k+1}.$$
Recall that the $\alpha_i$ are increasingly ordered. Thus $f\tau^{\nu+1}Q_{k+1}$ lives in $\ttau U_\alpha\ttau\wh\cH$, for
$$k+1-n\alpha_{k+2}-\gamma-\nu-1\leq\left((k+1)-n\alpha_{k+2}-\gamma)- (k-n\alpha_{k+1}-\gamma)\right)-1+\alpha\leq\alpha.$$
Now we should look at what happens to the class of $f\tau^{\nu+1}Q_{k+1}$ in the $\alpha$-graded piece of $\ttau\wh\cH$.

Note that $\left[f\tau^\nu Q_k\right]\neq0$ if and only if $\nu+n\alpha_{k+1}+\gamma-k+\alpha=0$, so
$$(z\tau\dd_\tau-\fe(\beta,\omega))f\tau^\nu Q_k=\big(z\tau\dd_\tau+(\beta-\alpha)z-\omega+\bar{\omega}z^2\big)(f)\tau^\nu Q_k-\tau^{\nu+1}Q_{k+1}=$$
$$=\big(z\tau\dd_\tau-2\Re(z_0\bar\omega)z-\omega+\bar{\omega}z^2\big)(f)\tau^\nu Q_k-\tau^{\nu+1}Q_{k+1}.$$

Now notice that $\tau$ divides $\tau\dd_\tau(f)$, so in fact $z\tau\dd_\tau(f)\tau^\nu Q_k\in\ttau U_{\alpha-1}\ttau\wh\cH$ and then we can further reduce our expression to
$$(z\tau\dd_\tau-\fe(\beta,\omega))f\tau^\nu Q_k=(-\omega-2\Re(z_0\bar\omega)z+\bar{\omega}z^2)f\tau^\nu Q_k-\tau^{\nu+1}Q_{k+1}.$$

On the other hand, $\tau^{\nu+1}Q_{k+1}$ does not vanish either in $\Gr_\alpha^{\ttau U} \ttau\wh\cH$ if and only if $\alpha_{k+2}=\alpha_{k+1}$. Indeed, we know that $\nu+n\alpha_{k+1}+\gamma-k+\alpha=0$, so doing the same as before, $k+1-n\alpha_{k+2}-\gamma-\nu-1=\alpha+n(\alpha_{k+2}-\alpha_{k+1})$ and the claim follows. Furthermore, in order to $(z\tau\dd_\tau-\fe(\beta,\omega))$ to vanish, we should impose that $\omega=0$, just by looking at the coefficients of the powers of $z$ in the expression for $f$.

Now if $k=n-1$, then everything would be the same as before except $-\tau^{\nu+1}Q_{k+1}$, which becomes $-(-n)^nt\tau^{\nu+1}$, whose class vanishes obviously in the graded piece under consideration.

In conclusion, $(z\tau\dd_\tau-\fe(\beta,\omega))^lf\tau^\nu Q_k$ can only vanish in $\Gr_\alpha^{\ttau U}\ttau\wh\cH$ if $\alpha=\beta$ (and then $\omega=0$), and does not do so until we get to an index $k+l$ such that $\alpha_{k+l}$ is strictly bigger than $\alpha_k$. Since there is a finite set of indexes, $(z\tau\dd_\tau+\alpha z)$ is nilpotent, of nilpotency index $n$ at most.

Condition iii' in \cite[\S 2.1.2.2]{Mo13} is equivalent to $z\tau\dd_\tau\ttau U_\alpha\ttau\wh\cH\subseteq\ttau U_\alpha\ttau\wh\cH$, using that $\ttau U_\alpha\ttau\wh\cH=\tau\ttau U_{\alpha+1}\ttau\wh\cH$, and such claim is a consequence from an argument very similar to the proof of condition v above. Finally, since $V_0\cR_{\cX}=\cO_{\ttau\cX}\langle z\dd_t,z\tau\dd_\tau\rangle$, it is clear from the computations above and the alternative expression for the filtration steps in Remark \ref{rem:tauVfilt} that they are cyclic $V_0\cR_{\cX}$-modules, and then coherent.

Summing up and noting that all the calculation was in fact independent of $z_0$, $\ttau\wh\cH$ is strictly $\RR$-specializable along $\ttau \cX_0$.
\end{proof}

Finally, we are able to state and prove our main result.
\begin{thm}\label{thm:HodgeData}
Let as before $\alpha_1,\ldots,\alpha_n$ be real numbers in $[0,1)$, increasingly ordered, and put $\cH=\cH(\alpha_i,\emptyset)$. For each $k=1,\ldots,n$, set $\rho(k)=-n\alpha_k+k$. Then the jumping numbers of the irregular Hodge filtration of $\cH$ are, up to an overall real shift, the numbers $\rho(k)$. The irregular Hodge numbers are the multiplicities of those jumping numbers, or in other words, the nonzero values of $|\rho^{-1}(x)|$, for $x$ real.

Moreover, recall that for all $r=0,\ldots,n-1$, we had $\nu_\alpha(r)=\lceil -\alpha+r-\gamma-n\alpha_{r+1}\rceil$, and that the operators $\bar{Q}_r$ were defined as
$$\bar{Q}_r=(-n)^r\prod_{i=1}^r(t\dd_t-\alpha_i).$$
Then, the irregular Hodge filtration $\Firr_\bullet\cH$ is given by
$$\Firr_{\alpha+j}\cH=\bigoplus_{k:j\geq\nu_\alpha(k)}\cO_X\bar{Q}_k.$$
\end{thm}

\begin{proof}
Since we know that $\wh\cH$ underlies an object in $\IrrMHM(\GG_{m,t})$  by Theorem \ref{thm:HypIrrMHM},
we conclude by \cite[Def. 2.52]{Sa15} that $\wh\cH$ is well-rescalable (as defined in [ibid., Def. 2.19]) and so we apply [ibid., Def. 2.22]. From Remark \ref{rem:tauVfilt}, we have
$$i_{\tau=z}^*\ttau V_\alpha\ttau\wh\cH=\ttau V_\alpha\ttau\wh\cH/(\tau-z)\ttau V_\alpha\ttau\wh\cH=\bigoplus_k\cO_{\cX}z^{\nu_\alpha(k)}\bar{Q}_k,$$
which is $z$-graded of finite rank.

Denote by $\pi$ the projection $\cX\ra X$. Then, the $z$-adic filtration on $\pi^*\cH[z^{-1}]$ induces a filtration on $i_{\tau=z}^*\ttau V_\alpha\ttau\wh\cH$, given by
$$F_ri_{\tau=z}^*\ttau V_\alpha\ttau\wh\cH:=\bigoplus_{s\leq r}\left(\bigoplus_{k\,:\,s\geq\nu_\alpha(k)}\cO_{X}\bar{Q}_k\right)z^s.$$
Then, $\Gr^F\left(i_{\tau=z}^*\ttau V_\alpha\ttau\wh\cH\right)$ is the Rees module associated to a new good filtration $\Firr_{\alpha+\bullet}$ on $\cH$, for some $k=0,\ldots,n-1$, which is the irregular Hodge filtration. More concretely, $\Firr_\bullet\cH$ is given by
$$\Firr_{\alpha+j}\cH=\bigoplus_{k\,:\,j\geq\nu_\alpha(k)}\cO_{X}\bar{Q}_k.$$
Therefore, its jumping numbers are $-\gamma+i-1-n\alpha_i$ for $i=1,\ldots,n$. Since the irregular Hodge filtration is defined up to an overall real shift, we can normalize the jumping numbers to $i-n\alpha_i$ and the irregular Hodge numbers will be their multiplicities.
\end{proof}

\bibliographystyle{amsalpha}
\bibliography{IrrHodgeHyper}

\def\cprime{$'$}
\providecommand{\bysame}{\leavevmode\hbox to3em{\hrulefill}\thinspace}
\providecommand{\MR}{\relax\ifhmode\unskip\space\fi MR }
\providecommand{\MRhref}[2]{%
  \href{http://www.ams.org/mathscinet-getitem?mr=#1}{#2}
}
\providecommand{\href}[2]{#2}
\begin{thebibliography}{CDRS18}

\bibitem[Ado94]{Adolphson}
Alan Adolphson, \emph{Hypergeometric functions and rings generated by
  monomials}, Duke Math. J. \textbf{73} (1994), no.~2, 269--290.

\bibitem[Ari10]{Arin}
D.~Arinkin, \emph{Rigid irregular connections on {$\mathbb{P}^1$}}, Compos.
  Math. \textbf{146} (2010), no.~5, 1323--1338.

\bibitem[BH06]{BorHor}
Lev~A. Borisov and R.~Paul Horja, \emph{Mellin-{B}arnes integrals as
  {F}ourier-{M}ukai transforms}, Adv. Math. \textbf{207} (2006), no.~2,
  876--927.

\bibitem[BMW18]{BMW}
Christine Berkesch, Laura~Felicia Matusevich, and Uli Walther, \emph{On
  normalized {H}orn systems}, preprint arXiv:1806.03355 [math.AG], 2018.

\bibitem[Bri70]{Brie}
Egbert Brieskorn, \emph{Die {M}onodromie der isolierten {S}ingularit\"aten von
  {H}yperfl\"achen}, Manuscripta Math. \textbf{2} (1970), 103--161.

\bibitem[CDRS18]{SevCastReich}
Alberto Casta{\~n}o~Dom{\'\i}nguez, Thomas Reichelt, and Christian Sevenheck,
  \emph{Examples of hypergeometric twistor $\mathcal{D}$-modules}, preprint
  arXiv:1803.04886 [math.AG], to appear in \emph{Algebra Number Theory}, 2018.

\bibitem[CG11]{CorGol}
Alessio Corti and Vasily Golyshev, \emph{Hypergeometric equations and weighted
  projective spaces}, Sci. China Math. \textbf{54} (2011), no.~8, 1577--1590.

\bibitem[CK99]{CK}
David~A. Cox and Sheldon Katz, \emph{Mirror symmetry and algebraic geometry},
  Mathematical Surveys and Monographs, vol.~68, American Mathematical Society,
  Providence, RI, 1999.

\bibitem[DE03]{AE}
Andrea D'Agnolo and Michael Eastwood, \emph{Radon and {F}ourier transforms for
  {$\mathcal{D}$}-modules}, Adv. Math. \textbf{180} (2003), no.~2, 452--485.

\bibitem[Dim04]{Di}
Alexandru Dimca, \emph{Sheaves in topology}, Universitext, Springer-Verlag,
  Berlin, 2004.

\bibitem[DL91]{DeLoe}
J.~Denef and F.~Loeser, \emph{Weights of exponential sums, intersection
  cohomology, and {N}ewton polyhedra}, Invent. Math. \textbf{106} (1991),
  no.~2, 275--294.

\bibitem[DS03]{DS}
Antoine Douai and Claude Sabbah, \emph{Gauss-{M}anin systems, {B}rieskorn
  lattices and {F}robenius structures. {I}}, Ann. Inst. Fourier (Grenoble)
  \textbf{53} (2003), no.~4, 1055--1116.

\bibitem[DS04]{DS2}
\bysame, \emph{Gauss-{M}anin systems, {B}rieskorn lattices and {F}robenius
  structures. {II}}, Frobenius manifolds, Aspects Math., E36, Vieweg,
  Wiesbaden, 2004, pp.~1--18.

\bibitem[DS13]{DettSa}
Michael Dettweiler and Claude Sabbah, \emph{Hodge theory of the middle
  convolution}, Publ. Res. Inst. Math. Sci. \textbf{49} (2013), no.~4,
  761--800.

\bibitem[ESY17]{SaEsYu}
H\'el\`ene Esnault, Claude Sabbah, and Jeng-Daw Yu, \emph{{$E_1$}-degeneration
  of the irregular {H}odge filtration}, J. Reine Angew. Math. \textbf{729}
  (2017), 171--227.

\bibitem[Fed18]{Fe}
Roman Fedorov, \emph{Variations of {H}odge structures for hypergeometric
  differential operators and parabolic {H}iggs bundles}, Int. Math. Res. Not.
  IMRN (2018), no.~18, 5583--5608.

\bibitem[GKZ94]{GKZbook}
Israel~M. Gel{\cprime}fand, Mikhail~M. Kapranov, and Andrei~V. Zelevinsky,
  \emph{Discriminants, resultants, and multidimensional determinants},
  Mathematics: Theory \& Applications, Birkh\"auser Boston Inc., Boston, MA,
  1994.

\bibitem[GMS09]{dGMS}
Ignacio~de Gregorio, David Mond, and Christian Sevenheck, \emph{Linear free
  divisors and {F}robenius manifolds}, Compositio Mathematica \textbf{145}
  (2009), no.~5, 1305--1350.

\bibitem[HS07]{HS1}
Claus Hertling and Christian Sevenheck, \emph{{N}ilpotent orbits of a
  generalization of {H}odge structures.}, J. Reine Angew. Math. \textbf{609}
  (2007), 23--80.

\bibitem[HTT08]{Hotta}
Ryoshi Hotta, Kiyoshi Takeuchi, and Toshiyuki Tanisaki,
  \emph{{$\mathcal{D}$}-modules, perverse sheaves, and representation theory},
  Progress in Mathematics, vol. 236, Birkh\"auser Boston, Inc., Boston, MA,
  2008, translated from the 1995 Japanese edition by Takeuchi.

\bibitem[Kat90]{Ka}
Nicholas~M. Katz, \emph{Exponential sums and differential equations}, Annals of
  Mathematics Studies, vol. 124, Princeton University Press, Princeton, NJ,
  1990.

\bibitem[KKP17]{KKP2}
Ludmil Katzarkov, Maxim Kontsevich, and Tony Pantev,
  \emph{Bogomolov-{T}ian-{T}odorov theorems for {L}andau-{G}inzburg models}, J.
  Differential Geom. \textbf{105} (2017), no.~1, 55--117.

\bibitem[Kou76]{Kouch}
Anatoli~G. Kouchnirenko, \emph{Poly\`edres de {N}ewton et nombres de {M}ilnor},
  Invent. Math. \textbf{32} (1976), no.~1, 1--31.

\bibitem[Mar18]{Mar18}
Nicolas Martin, \emph{Middle multiplicative convolution and hypergeometric
  equations}, preprint arXiv:1809.08867 [math.AG], 2018.

\bibitem[Moc11]{Mo5}
Takuro Mochizuki, \emph{Wild harmonic bundles and wild pure twistor
  {$D$}-modules}, Ast\'erisque (2011), no.~340, x+607.

\bibitem[Moc15a]{Mo13}
\bysame, \emph{Mixed twistor {$\mathcal{D}$}-modules}, Lecture Notes in
  Mathematics, vol. 2125, Springer, Cham, 2015.

\bibitem[Moc15b]{Mo15}
\bysame, \emph{Twistor property of {GKZ}-hypergeometric systems}, preprint
  arXiv:1501.04146 [math.AG], 2015.

\bibitem[{Rei}14]{Reich2}
Thomas {Reichelt}, \emph{Laurent polynomials, {GKZ}-hypergeometric systems and
  mixed {H}odge modules}, Compositio Mathematica \textbf{(150)} (2014),
  911--941.

\bibitem[RS15a]{ReiSe3}
Thomas Reichelt and Christian Sevenheck, \emph{Hypergeometric {H}odge modules},
  preprint arXiv:1503.01004 [math.AG], to appear in \emph{Algebraic Geometry},
  2015.

\bibitem[RS15b]{ReiSe}
\bysame, \emph{Logarithmic {F}robenius manifolds, hypergeometric systems and
  quantum $\mathcal{D}$-modules}, Journal of Algebraic Geometry \textbf{24}
  (2015), no.~2, 201--281.

\bibitem[RS17]{ReiSe2}
\bysame, \emph{Non-affine {L}andau-{G}inzburg models and intersection
  cohomology}, Ann. Sci. \'Ec. Norm. Sup\'er. (4) \textbf{50} (2017), no.~3,
  665--753.

\bibitem[Sab02]{Sa4}
Claude Sabbah, \emph{D\'eformations isomonodromiques et vari\'et\'es de
  {F}robenius}, Savoirs Actuels, EDP Sciences, Les Ulis, 2002, Math\'ematiques.

\bibitem[Sab06]{Sa2}
\bysame, \emph{Hypergeometric periods for a tame polynomial}, Port. Math.
  (N.S.) \textbf{63} (2006), no.~2, 173--226.

\bibitem[Sab08]{Sa8}
\bysame, \emph{{F}ourier-{L}aplace transform of a variation of polarized
  complex {H}odge structure.}, J. Reine Angew. Math. \textbf{621} (2008),
  123--158.

\bibitem[Sab18a]{Sa15}
\bysame, \emph{Irregular {H}odge theory \emph{(with the collaboration of
  {J}eng-{D}aw {Y}u)}}, M\'{e}m. Soc. Math. Fr. (N.S.) (2018), no.~156, vi+126.

\bibitem[Sab18b]{Sa18}
\bysame, \emph{Some properties and applications of {B}rieskorn lattices},
  Journal of Singularities \textbf{18} (2018), 239--248.

\bibitem[Sai88]{Saito1}
Morihiko Saito, \emph{Modules de {H}odge polarisables}, Publ. Res. Inst. Math.
  Sci. \textbf{24} (1988), no.~6, 849--995 (1989).

\bibitem[Sai90]{SaitoMHM}
\bysame, \emph{Mixed {H}odge modules}, Publ. Res. Inst. Math. Sci. \textbf{26}
  (1990), no.~2, 221--333.

\bibitem[Sch73]{Sch}
Wilfried Schmid, \emph{Variation of {H}odge structure: the singularities of the
  period mapping}, Invent. Math. \textbf{22} (1973), 211--319.

\bibitem[Sha18]{Shamoto}
Yota Shamoto, \emph{Hodge-{T}ate conditions for {L}andau-{G}inzburg models},
  Publ. Res. Inst. Math. Sci. \textbf{54} (2018), no.~3, 469--515.

\bibitem[Sim90]{Si2}
Carlos~T. Simpson, \emph{Harmonic bundles on noncompact curves}, J. Amer. Math.
  Soc. \textbf{3} (1990), no.~3, 713--770.

\bibitem[SY15]{Sa14}
Claude Sabbah and Jeng-Daw Yu, \emph{On the irregular {H}odge filtration of
  exponentially twisted mixed {H}odge modules}, Forum Math. Sigma \textbf{3}
  (2015), e9, 71 pp.

\bibitem[SY18]{SaYu18}
\bysame, \emph{Irregular {H}odge numbers of confluent hypergeometric
  differential equations}, preprint arXiv:1812.00755 [math.AG], 2018.

\bibitem[Yu14]{Yu}
Jeng-Daw Yu, \emph{Irregular {H}odge filtration on twisted de {R}ham
  cohomology}, Manuscripta Math. \textbf{144} (2014), no.~1-2, 99--133.

\end{thebibliography}

\end{document}